\documentclass[10pt]{amsart}
\usepackage[usenames]{color}
\usepackage{fullpage,amscd,amssymb,latexsym,url,hyperref}
\usepackage[all,2cell,ps]{xy}
\usepackage{fullpage}
\usepackage{graphicx}
\usepackage{epsfig,epstopdf}
\usepackage{float}
\usepackage{soul}
\usepackage{color}

\theoremstyle{plain}
\newtheorem{theorem}{Theorem}[section]

\newtheorem{lemma}[theorem]{Lemma}
\newtheorem{proposition}[theorem]{Proposition}

\theoremstyle{definition}

\newtheorem{remark}[theorem]{Remark}

\newtheorem*{remark*}{Remark}

\newtheorem*{problem*}{Problem}

\def\aut#1{\mathrm{Aut}(#1)}

\def\B{\mathfrak{B}}
\def\Z{\mathbb Z}
\def\G{\mathcal G}

\def\k{\xi}

\def\h{H}

\def\setof#1#2{\{#1\, : \,#2\}}
\def\t{t}
\def\r{r}

\usepackage{xspace}
\newcommand{\NC}{{(K)}\xspace}

\usepackage{xspace}
\newcommand{\nc}{{(R)}\xspace}

\newcommand{\Hol}{\operatorname{Hol}}
\newcommand{\Fix}{\operatorname{Fix}}

\newcommand{\dm}{\mathcal{D}}



\title{Skew braces of size  $p^2q$ I: abelian type}

\begin{document}

\author{E. Acri}
\author{M. Bonatto}

\address[E. Acri]{IMAS--CONICET and Universidad de Buenos Aires, 
Pabell\'on~1, Ciudad Universitaria, 1428, Buenos Aires, Argentina}
\email{eacri@dm.uba.ar}

\address[M. Bonatto]{Dipartimento di matematica e informatica - UNIFE}

\email{marco.bonatto.87@gmail.com}

\maketitle

\begin{abstract}
    This is the first part of a series of two articles. In this paper we enumerate and classify the left braces of size $p^2q$ where $p$ and $q$ are distinct prime numbers by the classification of regular subgroups of the holomorph of the abelian groups of the same order. We also provide the formulas that define the constructed braces.
\end{abstract}

\section{Introduction}

The study of the solutions to the set-theoretical Yang--Baxter equation (YBE) has started with \cite{MR1183474} as a discrete version of the braid equation. We say that for a given set $X$ and a function $r:X\times X\to X\times X$, the pair $(X,r)$ is a \emph{set-theoretical solution to the Yang--Baxter equation} if
\begin{equation}\label{YBE}
(id_X \times r)(r\times id_X)(id_X \times r)=(r\times id_X)(id_X \times r)(r\times id_X)
\end{equation}
holds.

A particular family of solutions is the family of non-degenerate solutions, i.e. solutions obtained as
\begin{equation}\label{non-deg}
r:    X\times X\longrightarrow X\times X,\quad (x,y)\mapsto (\sigma_x(y),\tau_y(x)),
\end{equation}
where $\sigma_x,\tau_x$ are permutations of $X$ for every $x\in X$. A particular case is given by \emph{non-degenerate involutive solutions}, that is $r^2=id_{X\times X}$. Non-degenerate solutions have been studied by several different authors \cite{MR1722951,MR1637256, MR1769723,MR1809284}. Rump introduced the notion of {\it (left) brace}, a binary algebraic structure providing examples of non-degenerate involutive solutions to YBE \cite{MR2132760}. These algebraic structures have been generalized later to \emph{skew (left) braces} by Guarnieri and Vendramin in \cite{MR3647970}, which provide non-involutive non-degenerate solutions.

A \emph{skew (left) brace} is a triple $(B,+,\circ)$ where $(B,+)$ and $(B,\circ)$ are groups (not necessarily abelian) such that
\[
a\circ(b+c)=a\circ b-a+a\circ c
\]
holds for every $a,b,c\in B$. Left braces, or braces for short, are skew braces for which the additive group is abelian. A skew brace $(B,+,\circ)$ is said to be a bi-skew brace if also $(B,\circ,+)$ is a skew brace \cite{biskew}.

The classification of skew braces is strictly related to the problem of finding non-degenerate solutions to \eqref{YBE}. Indeed, given an involutive non-degenerate solution as in \eqref{non-deg}, the permutation group generated by $\setof{\sigma_x}{x\in X}$ has a canonical structure of left brace and in \cite{MR3527540} it has been shown how to construct all the solutions with a given brace structure over the associated permutation group. Later, this approach was extended to skew braces in \cite{MR3835326}. Therefore, in a sense, the study of the solutions to \eqref{YBE} can be reduced to the classification of skew braces, since there is a way to construct all the solutions with a given permutation group endowed with a certain skew brace structure.

Recent progress on the classification problem for (skew) braces are, for instance: the classification of braces with cyclic additive group \cite{MR2298848, Rump}, skew braces of size $pq$ for $p,q$ different primes \cite{skew_pq}, braces of size $p^2q$ for $p,q$ primes with $q>p+1$ \cite{Dietzel}, braces of order $p^2,p^3$ where $p$ is a prime \cite{p_cube}, skew braces of order $p^3$ \cite{NZ} and skew braces of squarefree size \cite{squarefree}.

In this paper we enumerate the left braces of size $p^2q$ where $p$ and $q$ are distinct prime numbers and we also provide explicit formulas for braces. In a separate paper (\cite{second_paper}), we deal with the classification of skew braces of size $p^2q$ with non-abelian additive group.

Skew braces are also related to {\it Hopf-Galois extensions} as explained in \cite{Leandro-Byott}, since they are both in connection with regular subgroups of the holomorph of a group. The results of this paper have already been partially proved in \cite{Caranti} through the connection between skew braces and Hopf-Galois extensions (the authors claim that the missing cases are the subject of a second paper in preparation). 

We employ the same techniques we used in \cite{skew_pq}, which are based on the algorithm for the construction of skew braces with a given additive group developed in \cite{MR3647970}. Indeed, the algorithm allows to obtain all the skew braces with additive group $A$ from regular subgroups of its holomorph $\Hol(A)=A\rtimes\aut{A}$. The isomorphism classes of skew braces are parametrized by the orbits of such subgroups under the action by conjugation of the automorphism group of $A$ in $\Hol(A)$, \cite[Section 4]{MR3647970}.

The paper is organized as follows: in Section \ref{Sec Pre} we introduce all the necessary definitions and we explain the interplay between (skew) braces and regular subgroups and the techniques we are going to use to obtain the main results. In the subsequent sections we deal with the classification according to the relation between the primes $p$ and $q$. The enumeration of the left braces is collected in the following table:

\begin{table}[H]
	\centering
	\begin{tabular}{l|c|c}
		& $\#$ & Section \\
		\hline
		$p=1\pmod q,\ q>2$ & $\frac{q+15}{2}$ & \S\,\ref{section:p=1(q)}\\
		$p=1 \pmod q,\ q=2$ & $8$ & \S\,\ref{section:p=1(q)}\\
		$p=-1\pmod q$ & $5$ & \S\,\ref{section:p=-1(q)}\\
		$q=1\pmod p$, $q\neq 1\pmod{p^2}$ & $p+8$ & \S\,\ref{section:q=1(p)}  \\
		$q=1\pmod {p^2}$ & $2p+8$ & \S\,\ref{section:q=1(p^2)}\\
		$p=2$, $q\neq 1\pmod 4$ & $9$ & \S\,\ref{4p (i)}\\
		$p=2$, $q=1\pmod 4$ & $11$ & \S\,\ref{4p (ii)}\\
		$p$ and $q$ arithmetically independent & $4$ & \S\,\ref{section:alg_ind} \\
	\end{tabular}
	\caption{Enumeration of left braces of size $p^2q$.}
\end{table}

%\com{Check the enumeration. I don't get the number in section 3. I get $\frac{q+15}{2}$ if $q>2$. I checked with GAP 147=3*7*7 and 605=11*11*5 and the right formula should be $\frac{q+15}{2}$}

%\comment{Yeah, the enumeration is actually $\frac{q+15}{2}$ for $q>2$. For $q=2$ the enumeration is different because $|\B|=2$, right? If you feel to remove Section 3.3, let's try.}

In each section we also collect the enumeration results according to the algebraic properties of the braces in suitable tables.

In particular, our results provide an alternative proof to Conjectures 6.2-6.4 posed in \cite{MR3647970}, already proved in \cite{Smoktunowicz} and \cite{Dietzel}. Moreover, we prove that there are 2 quaternion left braces of size $4p$ for $p$ a prime number, providing a partial answer to \cite[Problem 2.28]{problems}. 

In \cite[Table 5.3]{MR3647970} we have the number of skew braces of order $n\leq 120$ with some exceptions. Due to a computational improvement of the algorithm calculating skew braces, in \cite{skew_trick} we have many tables collecting the number of skew braces of order $n\leq 868$ with some exceptions. We highlight that none of these exceptions are of the form $p^2q$ for different primes $p,q$. All these data are available in \cite{YBE}. This information was of invaluable help when working on the present paper and all of our results agree with those tables.

In \cite[Conjecture 4.1]{skew_trick}, the authors conjecture the number of skew braces of size $p^2q$ based on their tables. As an application of the results obtained in \cite{second_paper} and in the present work, we can give a positive answer to that conjecture, see \cite[\S 7]{second_paper}.

\section{Preliminaries}\label{Sec Pre}

A triple $(B,+,\circ)$ is said to be a \emph{skew (left) brace} if both $(B,+)$ and $(B,\circ)$ are (not necessarily abelian) groups and $$a\circ(b+c)=a\circ b-a+a\circ c$$ holds for every $a,b,c\in B$. Following the standard terminology for Hopf--Galois extensions, if $\chi$ is a group theoretical property, we say that a skew brace $(B,+,\circ)$ is of $\chi$-type if $(B,+)$ has the property $\chi$. Notice that, according to the standard terminology, skew (left) braces of abelian type are just called \emph{(left) braces}.

A {\it bi-skew brace} is a skew brace $(B,+,\circ)$ such that $(B,\circ,+)$ is also a skew brace (see \cite{biskew}). Equivalently
\begin{equation*}\label{eq for biskew}
	a+(b\circ c)=(a+ b)\circ a'\circ (a+ c) 
\end{equation*}
holds for every $a,b,c\in B$, where $a'$ denotes the inverse of $a$ in $(B,\circ)$.

Given a skew brace $(B,+,\circ)$, the mapping
\begin{equation*}\label{axiom}
	\lambda:(B,\circ)\to \aut{B,+}, \quad \lambda_a(b)=-a+a\circ b
\end{equation*}
is a homomorphism of groups and we denote its kernel as usual by $\ker \lambda$.

Let $(B,+,\circ)$ be a skew brace. A subgroup $I$ of $(B,+)$ is a {\it left ideal} of $B$ if $\lambda_a(I)\leq I$ for every $a\in B$. Every left ideal is a subgroup of $(B,\circ)$. If a left ideal $I$ is a normal subgroup of both $(B,+)$ and $(B,\circ)$ then $I$ is said to be an {\it ideal}. The set $$\Fix(B)=\{ a \in B : \lambda_x(a)=a \text{ for all $x\in B$} \}$$ is clearly a left ideal.

\begin{lemma}\label{Sylows are ideals} 
	Let p,q be primes with $p>2$ and let $(B,+,\circ)$ be a finite skew brace of size $p^2q$. 
	\begin{itemize}
		\item[(i)] If $p=\pm 1\pmod q$ then the unique $p$-Sylow subgroup of $(B,+)$ is an ideal. In particular, the $p$-Sylow subgroup of $(B,+)$ and $(B,\circ)$ are isomorphic.
		\item[(ii)] If $q= 1\pmod p$ then the unique $q$-Sylow subgroup of $(B,+)$ is an ideal.
	\end{itemize}
\end{lemma}

\begin{proof}
	The groups of size $p^2q$ with $p=\pm 1\pmod q$ has a unique $p$-Sylow subgroup, therefore it is characteristic. If we denote by $P$ the unique $p$-Sylow subgroup of $(B,+)$, we have that for every $x\in B$, $\lambda_x(P)=P$. Since it has order $p^2$, $P$ is the unique $p$-Sylow subgroup of $(B,\circ)$. Then it is an ideal. In particular $(P,+,\circ)$ is a skew brace of size $p^2$ and so according to the classification provided in \cite{p_cube} we have $(P,+)\cong (P,\circ)$. The same argument applies for (ii).
\end{proof}

The following lemma will be used to compute the formulas of the multiplicative operation of skew braces through all the paper.

\begin{lemma}\label{ker fix}
	Let $B$ be a skew brace. 
	\begin{itemize}
		\item[(i)]  If $\lambda_b(b)=b$ then 
		$$\underbrace{b\circ b\circ \ldots \circ b}_n=nb \quad \text{ and } \quad \lambda_{nb}=\lambda_b^n$$
		for every $n\in \mathbb{N}$ where $nb=\underbrace{b+ b+ \ldots + b}_n$.
		
		\item[(ii)]  If $a,c\in \ker{\lambda}$ and $b,d\in \Fix(B)$ then
		$$(a+b)\circ (c+d)=a+b+\lambda_b(c)+d.$$

	\end{itemize}
\end{lemma}{}

\begin{proof}{}
	(i) The statement follows by induction on $n$, since if $\lambda_b(b)=b$ then $b\circ b=b+\lambda_b(b)=b+b=2b$ and $\lambda_{2b}=\lambda_{b\circ b}=\lambda_b^2$.

	(ii) Let $a,c\in \ker{\lambda}$ and $b,d\in \Fix(B)$. Then
	$$(a+b)\circ (c+d)=a+b+\lambda_{a\circ b}(c+d)=a+b+\lambda_b(c+d)=a+b+\lambda_b(c)+d$$ where we are using that $\lambda_{a+b}=\lambda_{a\circ b}=\lambda_b$ since $a\in \ker{\lambda}$.
\end{proof}

\subsection{Skew braces and regular subgroups}

Recall that a group of permutations $G$ acting on a set $X$ is said to be \emph{regular} if, for any $x,y\in X$, there exist a unique $g\in G$ such that $g(x)=y$.  The holomorph $\Hol(A) =A\rtimes \aut{A}$ of a group $A$ can be embedded in the group of permutations on $A$. In fact, an element $(a,f)\in\Hol(A)$ acts on $A$ by 
\begin{equation}\label{canonical action}
(a,f)\cdot x=a+f(x)    
\end{equation}
 for all $x\in A$ where the operation on $A$ is written additively. So, a subgroup $G$ of $\Hol (A)$ is said to be \emph{regular} if the image of $G$ in the group of permutations on $A$ is regular. 

In \cite{MR3647970}, it has been shown that given a group $(A,+)$ there exists a bijective correspondence between isomorphism classes of skew braces $(A,+,\circ)$ and orbits of regular subgroups of $\Hol(A,+)$ under conjugation by $\aut{A,+}$ identified with the subgroup $1\times \aut{A}\leq \Hol(A,+)$. 

Let 
$$\pi_1:\Hol(A)\longrightarrow A, \quad (a,f)\mapsto a \quad\text{and}\quad \pi_2:\Hol(A)\longrightarrow \aut{A},\quad (a,f)\mapsto f$$ be the canonical surjections.

\begin{theorem}\cite[Theorem 4.2, Proposition 4.3]{MR3647970}\label{thm:skew_holomorph}
	Let $(A,+)$ be a group. If $\circ$ is an operation such that $(A,+,\circ)$ is a skew brace, then $\{ (a,\lambda_a): a\in A \}$ is a regular subgroup of $\Hol(A,+)$. Conversely, if $G$ is a regular subgroup of $\Hol(A,+)$, then $A$ is a skew brace with $$a\circ b=a+f(b)$$ where $(\pi_1|_G)^{-1}(a)=(a,f)\in G$ and $(A,\circ)\cong G$.
	
	Moreover, isomorphism classes of skew braces over $A$ are in bijective correspondence with the orbits of regular subgroups of $\Hol(A)$ under the action of $\aut{A}$ by conjugation.
\end{theorem}

We can construct all skew braces with additive group isomorphic to $A$ by finding representatives of the orbits of regular subgroups under the action by conjugation of $\aut{A}$ on $\Hol(A)$ using Theorem \ref{thm:skew_holomorph}. 

The connection between a regular subgroup and the multiplicative group structure of the associated skew brace is the following.

\begin{remark}\label{remark for lambdas}
	Let $(A,+)$ be a group and $G$ a regular subgroup of $\Hol(A)$. According to Theorem \ref{thm:skew_holomorph}, $(A,+,\circ)$ where 
	\[
	a\circ b=a+\pi_2((\pi_1|_G)^{-1}(a))(b)
	\]
	for every $a,b\in A$ is a skew brace. In other words, $\lambda_a=\pi_2((\pi_1|_G)^{-1}(a))$ and 
	so if $G$ is finite $|\ker{\lambda}|=\frac{|G|}{|\pi_2(G)|}$.

\end{remark}

\subsection{Regular subgroups}

In this section we describe a method to search for regular subgroups of the holomorph of a given finite group $A$ and we fix some notation and terminology. Such method is inspired by \cite[Section 2.2]{NZpaper}. As above, let 
$$\pi_1:\Hol(A)\longrightarrow A, \quad (a,f)\mapsto a \quad\text{and}\quad \pi_2:\Hol(A)\longrightarrow \aut{A},\quad (a,f)\mapsto f$$ be the canonical surjections.
The map $\pi_1$ controls the regularity of a subgroup $G\leq \Hol(A)$ with $|G|=|A|$. 
\begin{lemma}\label{regularity}
	Let $A$ be a finite group and $G\leq \Hol(A)$. The following are equivalent:
	\begin{enumerate}
		\item[(i)] $G$ is regular.
		\item[(ii)] $|A|=|G|$ and $\pi_1(G)=A$.
		\item[(iii)] $|A|=|G|$ and $G\cap \left(1\times \aut{A}\right)=1$.
	\end{enumerate}
\end{lemma}

%\begin{proof}
%    It is straightforward to verify that $G$ is a regular subgroup of $\Hol(A)$ if and only if the map $\theta: G\times A \to A\times A$ given by $(g,x)\mapsto (g\cdot x,x)$ is a bijection. First, we see that $(ii)$ and $(iii)$ are equivalent. From $(ii)$, we have that $\pi_1$ is a bijection. So, if we assume that $(1,f)\in G$, we have $\pi_1(1,f)=1=\pi_1(1,1)$ and therefore $f=1$. For the converse, if $\pi_1$ is not a bijection, then there exist two elements $(a,f)$ and $(a,h)$ in $G$ with $f\neq h$. So $(a,f)^{-1}(a,h)=(1,f^{-1}g)\in G$, a contradiction.
%   Now we see that $(i)$ implies $(ii)$. Since $\theta$ is a bijection, we have $|A|=|G|$. For every $a\in A$, we have that $(a,1)=\theta(g,x)=(g\cdot x,x)$ for certain $(g,x)\in G$. So $x=1$ and $a=g\cdot 1=\pi_1(g)$. For the converse, since $|A|=|G|$, it is sufficient to show that $\theta$ is surjective. Given $x,y\in A$, there exist unique elements $g,h\in G$ such that $\pi_1(g)=x$ and $\pi_1(h)=y$. Since $\pi_1(g)=g\cdot 1$ and $\pi_1(h)=h\cdot 1$, we have that $hg^{-1}\cdot x=y$ and then $\theta(hg^{-1},x)=(y,x)$, so $\theta$ is surjective, therefore bijective.
%\end{proof}

%\comment{Alternative shorter proof}
\begin{proof}
%Let $\theta$ be the canonical faithful action of $G$ on $A$, defined by $\theta(g,f)(a)=gf(a)$ for every $a,g\in A$ and $f\in \aut{A}$. %Let $(g,f)\in \ker{\theta}$. Then $gf(1)=g=1$ and thus $f(a)=a$ for every $a\in A$, i.e. $f=1$. 
Let $\theta$ be the canonical action defined in \eqref{canonical action}. Since the action $\theta$ is faithful and $A$ is finite we have that the following are equivalent:
	\begin{enumerate}
		\item[(i)] $G$ is regular.
		\item[(ii)] $|A|=|G|$ and the orbit of $1$ under $\theta$ is $A$.
		\item[(iii)] $|A|=|G|$ and the stabilizer of $1$ under the action of $G$ is trivial.
	\end{enumerate}

The stabilizer of $1$ under the action $\theta$ is $G\cap (1\times \aut{A})$ and the orbit of $1$ is $\pi_1(G)$. Thus, the claim follows.
\end{proof}

In order to show that a subgroup $G$ of the holomorph of $A$ is regular we are using Lemma \ref{regularity} or we are checking directly that the stabilizer of the identity is trivial. We will do it explicitly in Proposition \ref{cyclic no pq} and Theorem \ref{cor_formula_abelian_pp} and in the sequel we will omit the computations.

We are going to apply the following two-step strategy to compute a set of representative of conjugacy classes of regular subgroups of $\Hol(A)$.

The first step is to provide a list of non-conjugate regular subgroups of the holomorph, according to some properties invariant under conjugation by elements of the subgroup $1\times \aut{A}$ in $\Hol(A)$. Let $G$ be a subgroup of $\Hol(A)$, then the conjugacy class of $\pi_2(G)$ in $\aut{G}$ and so in particular $|\pi_2(G)|$ is an invariant. The map $\pi_2$ is a group homomorphism and 
the size of the image of $\pi_2$ divides both $|(A,+)|$ and $|\aut{A,+}|$, so it divides their greatest common divisor. If the image under $\pi_2$ is trivial, the associated skew brace is the unique trivial skew brace over $A$.

If two groups $G$ and $H$ have the same image under $\pi_2$ and they are conjugate by $h$, then $h$ normalizes their image under $\pi_2$ and $\ker{\pi_2|_G}=h (\ker{\pi_2|_H}) h^{-1}$. Then the kernel with respect to $\pi_2$ can be chosen up to the action of the normalizer of the image of $\pi_2$. Therefore, for the first step we need to:

\begin{itemize}
	\item Find the greatest common divisor between $|A|$ and $|\aut{A}|$.
	\item For each $k$ dividing such greatest common divisor, find the conjugacy classes of subgroups of size $k$ of $\aut{A,+}$. In case $\aut{A}$ is abelian, we have to compute all its subgroups of the desired size.
	\item The kernel of $\pi_2$ is a subgroup of $\Hol(A)$ and it is contained in $A\times 1$. So we need to find a set of representatives of the orbits of subgroups of size $\frac{|A|}{k}$ of $A$, under the action of the normalizer of the image of $\pi_2$ on $A$. %\comment{$\ker{\pi_2}$ is clearly a subgroup of $A$, without using Lemma \ref{kernel}}
\end{itemize}

In order to show that two groups with the same image under $\pi_2$ and the same kernel are not conjugate, we need to introduce some \emph{ad hoc} arguments, which depends on the particular structure of $A$ and its automorphisms. The isomorphism class of the groups in the list of representatives of conjugacy classes can be easily deduced using the list of groups of size $p^2q$ and so we omit a proof of such isomorphisms.

The second step is to show that any regular subgroup is conjugate to one in the list computed in the first step. In particular, we need to describe the subgroups of $\Hol(A)$ according to the possible values of the invariants mentioned above. 

Let $\setof{\alpha_i}{1\leq i\leq n}$ be a set of generators of $\pi_2(G)$ and $\setof{k_j}{1\leq j\leq m}$ be a set of generators of the kernel of $\pi_2|_G$. The regular subgroup $G$ can be presented as follows:
$$G=\langle k_1,\ldots, k_m, u_1\alpha_1,\ldots, u_n\alpha_n \rangle,$$
for some $u_i\in A$. According to Lemma \ref{regularity}, $u_i\neq 1$, since $G$ is regular. We will call this presentation the {\it standard presentation} of $G$. Note that we can choose any representative of the coset of $u_i$ with respect to the kernel, without changing the group $G$ (this is equivalent to multiply on the left the elements $u_i\alpha_i$ by elements of the kernel). Moreover, any generator can be multiplied by any element in $G$, without changing the group. We will use these operations to make computations easier without further explanation.

The group $G$ has to satisfy the following necessary conditions, that will provide constrains over the choice of the elements $u_i$: 
\begin{itemize}
	
	\item[\NC] The kernel of $\pi_2|_G$ is normal in $G$.
	
	\item[\nc] The generators $\setof{u_i\alpha_i}{1\leq i\leq n}$ satisfy the same relations as $\setof{\alpha_i}{1\leq i\leq n}$ modulo $\ker{\pi_2|_G}$ (e.g. if $\alpha_i^n=1$ then $(u_i\alpha_i)^n\in \ker{\pi_2|_G}$).
	
\end{itemize}

Given one of such groups, we can conjugate by the elements of the normalizer of $\pi_2(G)$ in $\aut{A}$ that stabilizes the kernel of $\pi_2|_G$ in order to show that $G$ is conjugate to one of the chosen representatives.

\subsection{Notation}\label{groups and auto}
 
In the paper the symbols $p$ and $q$ denote prime numbers and we always assume that they are distinct. We denote by $\Z_n^\times$ the group of units of the multiplicative group $\Z_n$. Through all the paper we are using the following presentations for the two abelian groups of size $p^2q$:

\begin{itemize}
	
	\item[(i)] $\mathbb{Z}_{p^2q}=\langle \sigma, \tau\ |\ \sigma^{p^2}=\tau^q=1,\ [\sigma,\tau]=1 \rangle$,
	\item[(ii)] $\mathbb{Z}_{p}^2\times \mathbb{Z}_q=\langle \sigma, \tau, \epsilon\ |\ \sigma^p=\tau^p=\epsilon^q=1,\ [\sigma,\tau]=[\sigma,\epsilon]=[\tau,\epsilon]=1 \rangle$.
\end{itemize}
The automorphism group of $\Z_{p^2q}$ is isomorphic to $\mathbb{Z}_{p^2}^\times\times \mathbb{Z}_q^\times$ and it has size $p(p-1)(q-1)$. We denote by $\varphi_{i,j}$ the automorphism defined by $$\sigma\mapsto \sigma^i \qquad\text{and}\qquad \tau\mapsto \tau^j$$ where $i\in \mathbb{Z}_{p^2}^\times$ and $j\in \mathbb{Z}_{q}^\times$. 

The automorphism group of $\Z_p^2\times \Z_q$ is isomorphic to $GL_{2}(p)\times \Z_q^\times$ and it has size $p(p-1)^2(p+1)(q-1)$. We can identify the automorphisms of $\Z_p^2\times \Z_q$ as $M\alpha$ where $M$ is an invertible matrix in the basis $\sigma,\tau$ and $\alpha\in \mathbb{Z}_q^\times$.

Note that, in both cases, the Sylow subgroups are characteristic and therefore also normal in the holomorph.

If $C$ is a cyclic group acting on a group $G$ by $\rho:C\longrightarrow \aut{G}$ and $\rho(1)=f$, then $G\rtimes_{f} C$ denotes the semidirect product determined by the action $\rho$ (which is uniquely determined by the image of the generators of $C$ under $\rho$).

\section{Left braces of size $p^2q$ with $p=1\pmod{q}$}\label{section:p=1(q)}

%\comment{We first say that the case $q=2$ in included and then we say that there is a special section about it. We need to explain better what we mean}\com{I can handle this.}

%\com{I realized that in this section, we don't need to split cases for $q=2$. Maybe we kept this because in the other paper we needed. The only difference is Prop 3.8 and 3.9. We can put them in the section for $\Z_p^2\times\Z_q$ and put off subsection 3.3, so it makes sense that "including the case $q=2$, unless stated otherwise.". We have to make some comments about $\B$ because if $q=2$, then $\frac{q+3}{2}$ makes no sense. Also we can comment that $\Z_{p^2}\rtimes_t\Z_q$ is the dihedral group if $q=2$. If you agree, I can make the changes and delete section 3.3}

In this section we assume that $p$ and $q$ are prime numbers such that $p=1\pmod{q}$, including the case $q=2$ unless stated otherwise. Crespo in \cite[\S 5.1 and \S 5.3]{Crespo_2p2} classified the left braces of order $2p^2$. We will obtain her results as a particular case. We denote by $\B$ the subset of $\Z_q$ that contains $0,1,-1$ and one out of $k$ and $k^{-1}$ for $k\neq 0,1,-1$. For $q=2$ we define $\B=\{0,1\}$. Notice that for $q>2$, we have $|\B|=\frac{q+3}{2}$. We denote by $g$ a fixed element of order $q$ in $\Z_{p}^\times$ and by $\t$ a fixed element of order $q$ in $\Z_{p^2}^\times$. 
%In Subsection \ref{2pp}, we deal with the special case $q=2$.

According to \cite[Proposition 21.17]{EnumerationGroups}, the non-abelian groups of size $p^2q$ are the following:
\begin{itemize}
	\item[(i)] $\mathbb{Z}_{p^2}\rtimes_{\t} \mathbb{Z}_q\cong\langle \sigma, \tau\ |\ \sigma^{p^2}=\tau^q=1,\ \tau\sigma\tau^{-1}=\sigma^{\t} \rangle$.
	\item[(ii)] $\G_{k}=\langle \sigma,\tau,\epsilon\,|\, \sigma^p=\tau^p=\epsilon^q=1,\, \epsilon \sigma \epsilon^{-1}=\sigma^g,\, \epsilon \tau \epsilon^{-1}=\tau^{g^k}\rangle \cong\mathbb{Z}_{p}^2 \rtimes_{\dm_k} \mathbb{Z}_q$, where $\dm_k$ is the diagonal matrix with diagonal entries $g, g^k$ for $k\in\B$.
\end{itemize}
Notice that for $q=2$, the group in (i) is isomorphic to the dihedral group of order $2p^2$ and we can assume $t=-1$. The groups in item (ii) are parametrized by the set $\B$ since $\G_k$ and $\G_{k^{-1}}$ are isomorphic if $k\neq 0,1,-1$. In particular, there are $\frac{q+3}{2}$ isomorphism classes of such groups if $q>2$ and $2$ isomorphism classes if $q=2$. We will take into account such isomorphism for the enumeration of skew braces according to the isomorphism class of their additive and multiplicative groups.

The following table collects the enumeration of left braces according to their additive and multiplicative groups.

\begin{table}[H]
	\centering
	\begin{tabular}{c|c|c|c|c|c}
		$+ \backslash \circ$ & $\Z_{p^2q}$ & $\Z_{p^2}\rtimes_{\t}\Z_q$ &$\Z_p^2\times\Z_q$ &   $\G_k$, $k\in \B\setminus\{2\}$ & $\G_{2}$\\
		\hline
		$\Z_{p^2q}$ & $2$ & $1$ & - &-   &- \\
		$\Z_p^2\times\Z_q$ &- &- & $2$   & $1$ & $2$	
	\end{tabular}
	\caption{Enumeration of left braces of size $p^2q$ with $p=1\pmod{q}$.}
\end{table}
Note that for $q=3$ we have that $\B=\{0,1,-1\}$, $2=-1$ and so in such case we have $2$ skew braces of $\Z_p^2\times \Z_q$-type with multiplicative group isomorphic to $\G_{-1}$. For the case $q=2$, we have $2=0$ and then $\B\setminus \{2\}=\{1\}$. %\com{Add a comment about $q=2$.}

\subsection{Left braces of cyclic type}\label{subsection:p=1(q)_cyclic}

In this section we will denote by  $A$ the cyclic group $\Z_{p^2q}$ and we employ the notation as in Section \ref{groups and auto} for the generators of the group and its automorphisms.
If $G$ is a regular subgroup of $\Hol(A)$, then $|\pi_2(G)|$ belongs to $\{1,q,p,pq\}$, since it divides $|\aut{A}|=p(p-1)(q-1)$ and $p^2q$. First, we see that there are not regular subgroups with $|\pi_2(G)|=pq$.

\begin{proposition}\label{cyclic no pq}
	Let $G$ be a regular subgroup of $\Hol(A)$. Then $|\pi_2(G)|\neq pq$.
\end{proposition}

\begin{proof}
	The unique subgroup of $\aut{A}$ of size $pq$ is generated by $\alpha=\varphi_{g(p+1),1}$ and the unique subgroup of order $p$ of $A$ is $\langle \sigma^p \rangle$. Assume that $G$ is a subgroup of size $p^2q$ of $\Hol(A)$ with $\pi_2(G)=\langle \alpha\rangle$. Then, the standard presentation of $G$ is
	\[
	G=\langle \sigma^p,\ \sigma^a\tau^b\alpha\rangle
	\]
	for some $a,b$.	Then, since $(p+1)^m=pm+1\pmod{p^2}$ we have that
	$$(\sigma^a\tau^b\alpha)^q=\sigma^{a\frac{g^q(p+1)^q-1}{g(p+1)-1}}\tau^{bq}\alpha^q=\sigma^{a\frac{qp}{g(p+1)-1}}\alpha^q\in G.$$
	Hence, since $\sigma^p\in G$ then $\alpha^q\in G\cap (1\times \aut{A})$ and so, according to Lemma \ref{regularity}, $G$ is not regular.
\end{proof}

The following theorems hold under the more general condition that $q\neq 1\pmod{p}$ and it will be used also in the sequel. %In particular, in these theorems $q=2$ is allowed.

\begin{theorem}\label{prop:cyclic_ker=pq}
	Let $p,q$ be primes with $q\neq 1\pmod{p}$. 
	The unique left brace of cyclic type with $|\ker{\lambda}|=pq$ is $(B,+,\circ)$ where 
	\begin{equation}\label{eq:formula cyclic pq}
	\begin{pmatrix} n \\ m \end{pmatrix} + \begin{pmatrix} s \\ r \end{pmatrix} =\begin{pmatrix}  n+s \\ m+r \end{pmatrix}, \qquad\qquad \begin{pmatrix} n \\ m \end{pmatrix} \circ \begin{pmatrix} s \\ r \end{pmatrix} =\begin{pmatrix}  n+s+pns \\ m+r \end{pmatrix}
	\end{equation}
	for every $0\leq n,s\leq p^2-1$ and $0\leq m,r\leq q-1$. In particular, $(B,\circ)\cong \mathbb{Z}_{p^2q}$.
\end{theorem}

\begin{proof}
	The unique subgroup of size $p$ of $\aut{A}$ is the subgroup generated by $\varphi_{p+1,1}$ and $A$ has a unique subgroup of size $pq$, namely $\langle \sigma^p, \tau\rangle$.
	Assume that $G$ is a regular subgroup of $\Hol(A)$ with $|\pi_2(G)|=p$. Then $G$ has the following standard presentation:
	\[
	G=\langle\sigma^p,\ \tau,\ \sigma^a\tau^b\varphi_{p+1,1}\rangle=\langle\sigma^p,\ \tau,\ \sigma^a\varphi_{p+1,1}\rangle,
	\]
	where $1\leq a\leq p-1$ and in particular $G$ is abelian.
	According to \cite[Corollary 4.3]{nilpotent_type}, since both the additive and the multiplicative group of the left brace associated to $G$ are abelian, then such left brace splits as a direct product of the left brace of size $p^2$ and a trivial left brace of size $q$. According to the classification of left braces of size $p^2$ given in \cite[Proposition 2.4]{p_cube}, there is just one non-trivial such brace with cyclic additive group and it leads to formula \eqref{eq:formula cyclic pq}. The group $(B,\circ)$ is cyclic, since according to Lemma \ref{Sylows are ideals} its $p$-Sylow subgroup is cyclic.
\end{proof}

\begin{theorem}\label{prop:cyclic_ker=p2}
	Let $p,q$ be primes with $q\neq 1\pmod{p}$. The unique left brace of cyclic type with $|\ker{\lambda}|=p^2$ is $(B,+,\circ)$ where
	\begin{eqnarray}\label{formula cyclic pp}
	\begin{pmatrix} n \\ m \end{pmatrix} + \begin{pmatrix} s \\ r \end{pmatrix} =\begin{pmatrix}  n+s \\ m+r \end{pmatrix}, \qquad\qquad \begin{pmatrix} n \\ m \end{pmatrix} \circ \begin{pmatrix} s \\ r \end{pmatrix} =\begin{pmatrix}  n+\t^ms \\ m+r \end{pmatrix}.
	\end{eqnarray}
	for every $0\leq n, s\leq p^2-1$ and $0\leq m, r\leq q-1$. In particular, $(B,\circ)\cong \Z_{p^2}\rtimes_{\t}\Z_q$ and $(B,+,\circ)$ is a bi-skew brace
\end{theorem}
\begin{proof}
	According to \cite[Corollary 1.2]{skew_pq}, we have that \eqref{formula cyclic pp} defines a bi-skew brace with the desired properties. We prove that there exists just one such left brace. The unique subgroup of order $q$ of $\aut{A}$ is the subgroup generated by $\varphi_{\t,1}$. The unique subgroup of order $p^2$ of $A$ is generated by $\sigma$. Let $G$ be a regular subgroup with $|\pi_2(G)|=q$, then $G$ has the following standard presentation:
	\[
	G_b=\langle\sigma,\ \tau^b\varphi_{\t,1}\rangle,
	\]
	with $1\leq b\leq q-1$.
	The group $G_b$ is conjugate to the subgroup 
	$$H=\langle \sigma,\tau\varphi_{\t,1}\rangle \cong \mathbb{Z}_{p^2}\rtimes_{\t} \mathbb{Z}_q$$
	by $\varphi_{1,b}^{-1}$. Therefore there exist a unique left brace with the desired properties up to isomorphism.
\end{proof}

In the following table we summarize the results of this subsection. The columns refer to the isomorphism class of the multiplicative group.

\begin{table}[H]
	\centering
	\begin{tabular}{c|c|c}
		$|\ker{\lambda}|$ &  $\mathbb{Z}_{p^2q}$  &  $\mathbb{Z}_{p^2}\rtimes_{\t} \mathbb{Z}_q$  \\
		\hline
		$pq$ & $1$ &-  \\
		$p^2$ &- & $1$ \\
		$p^2q$ & $1$ &-
	\end{tabular}
	\caption{Enumeration of left braces of cyclic type of size $p^2q$ for $p=1\pmod{q}$.}
	\label{table:cyclic_p=1}
\end{table}

\subsection{Left braces of $\Z_{p}^2\times \Z_q$-type}\label{subsection:p=1(q)_abelian_non-cyclic}

In this section $A$ will denote the abelian group $\Z_{p}^2\times \Z_q$. If $G$ is a regular subgroup of $\Hol(A)$ then $|\pi_2(G)|$ divides both $p^2q$ and $|\aut{A}|=p(p-1)^2(p+1)(q-1)$. So $|\pi_2(G)|$ divides $pq$ since $p=1\pmod{q}$.

\begin{remark}\label{action on q-Syl trivial}
	The automorphism of $A$ of order $p,q$ or $pq$ acts trivially on its $q$-Sylow subgroup, since $p$ and $q$ does not divide $q-1$. Hence, if $G$ is a regular group of size $p^2q$ then the action of $\pi_2(G)$ on the $q$-Sylow subgroup of $A$ is trivial.
\end{remark}

\begin{remark}\label{subgroups of GL} %\comment{we are using that $p=1\pmod{q}$} 
	Let assume that $p$ and $q$ are prime numbers such that $p=1\pmod{q}$ and consider	$$C=\begin{bmatrix} 1& 1 \\ 0 & 1\end{bmatrix},\quad \dm_s=\begin{bmatrix} g & 0\\ 0 & g^s \end{bmatrix} \quad\text{and}\quad \widetilde\dm=\begin{bmatrix} 1 & 0\\ 0 & g \end{bmatrix}$$
	for $0\leq s\leq q-1$.
	It is a well-known fact that, up to conjugation, the subgroups of order $q$ of $GL_2(p)$ are generated by one of the matrices $\dm_s$	for $s\in\B$ (including the case $q=2$). Notice that $\widetilde \dm$ is conjugated to $\dm_0$. 
	
	The subgroups of order $p$ of $GL_2(p)$ are the $p$-Sylow subgroups, and then they are all conjugate to the group generated by $C$.

	A set of representatives of conjugacy classes of subgroups of order $pq$ of $GL_2(p)$ is given by $H_s=\langle C,\dm_s\rangle$ and $\widetilde{H}=\langle C,\widetilde \dm\rangle$, where $0\leq s\leq q-1$.
	Indeed, up to conjugation, we can assume that the $p$-Sylow subgroup of a group $H$ of size $pq$ is generated by $C$. Then the generator of order $q$ is an upper triangular matrix, since it is in the normalizer of $C$ and then we can assume that $H$ is generated by $C$ and by a diagonal matrix of order $q$ which can be taken to be $\dm_s$ for $0\leq s\leq q-1$ or $\widetilde \dm$. Such matrices are not conjugate by the elements of the normalizer of $C$ and so the correspondent subgroups are not conjugate.  
\end{remark}

In the following theorem we assume once more that $q\neq 1\pmod{p}$ as we did in Theorem \ref{prop:cyclic_ker=pq}.% \com{In particular, in these theorems $q=2$ is allowed again.}

\begin{theorem}\label{sub_abelian_p}
	Let $p,q$ be primes such that $q\neq 1\pmod{p}$.
	The unique left brace of $A$-type with $|\ker{\lambda}|=pq$ is $(B,+,\circ)$ where
	\begin{eqnarray}\label{formula_sub_abelian_p}
	\begin{pmatrix} x_1 \\ x_2\\x_3 \end{pmatrix} + \begin{pmatrix} y_1 \\ y_2\\y_3 \end{pmatrix} =\begin{pmatrix} x_1+y_1 \\ x_2+y_2\\ x_3+y_3 \end{pmatrix},\qquad 
	\begin{pmatrix} x_1 \\ x_2 \\x_3\end{pmatrix} \circ \begin{pmatrix} y_1 \\ y_2 \\y_3 \end{pmatrix} =\begin{pmatrix} x_1+y_1+x_2 y_2 \\x_2+y_2 \\ x_3+y_3\end{pmatrix}\end{eqnarray}
	for every $0\leq x_1,x_2,y_1,y_2\leq p-1,\, 0\leq x_3,y_3\leq q-1$. In particular, $(B,\circ)\cong \mathbb{Z}_p^2\times \mathbb{Z}_q$.
\end{theorem}

\begin{proof}
	Let $G\leq \Hol(A)$ be a regular subgroup and let $|\pi_2(G)|=p$. According to Remarks \ref{action on q-Syl trivial} and \ref{subgroups of GL} and the fact that $q\neq 1\pmod p$ we can assume that $\pi_2(G)$ is generated by $C$, up to conjugation. Therefore the group $G$ has the following standard presentation: $$G=\langle v,\ \epsilon,\  \sigma^a \tau^b C\rangle$$
	for some $0\leq a,b\leq p-1$ and $v\in\langle\sigma,\tau\rangle$. The condition \NC implies that $v\in \langle\sigma\rangle$ and then $G$ is abelian.
	Therefore both the additive and the multiplicative group of the left braces obtained by $G$ are abelian, and so according to \cite[Corollary 4.3]{nilpotent_type} such left braces split as a direct product of a left brace of size $p^2$ and a trivial left brace of size $q$. According to the classification of left braces of size $p^2$ given in \cite[Proposition 2.4]{p_cube}, there is just one non-trivial such brace and it leads to formula \eqref{formula_sub_abelian_p}. According to Lemma \ref{Sylows are ideals}, the group $(B,\circ)$ is isomorphic to $A$ since its $p$-Sylow subgroup is elementary abelian.
\end{proof}

%\com{In the following theorem, $p$ and $q$ are primes without further conditions since every subgroup of order $q$ of $\aut{A}$ is included in $\aut{\Z_p^2}$, so $q$ doesn't play any role.}

%\comment{Actually we are using that $p=1\pmod{q}$, otherwise the automorphisms $\dm_s$ are not there. }

\begin{theorem}\label{cor_formula_abelian_pp}
	Let $p,q$ be primes such that $p=1\pmod{q}$. The left braces of $A$-type with $|\ker{\lambda}|=p^2$ are $B_s=(A,+,\circ)$ for $s\in\B$ where
	\begin{eqnarray}\label{formula1}
	\begin{pmatrix} x_1 \\ x_2\\x_3 \end{pmatrix} + \begin{pmatrix} y_1 \\ y_2\\y_3 \end{pmatrix} =\begin{pmatrix} x_1+y_1 \\ x_2+y_2\\x_3+y_3 \end{pmatrix},\qquad 
	\begin{pmatrix} x_1 \\ x_2 \\x_3\end{pmatrix} \circ \begin{pmatrix} y_1 \\ y_2 \\y_3 \end{pmatrix} =\begin{pmatrix} x_1+g^{x_3} y_1 \\x_2+\left(g^{s}\right)^{x_3} y_2 \\ x_3+y_3\end{pmatrix}\end{eqnarray}
	for every $0\leq x_1,x_2,y_1,y_2\leq p-1,\, 0\leq x_3,y_3\leq q-1$. In particular $(B_s,\circ)\cong \G_s$, and $B_s$ is a bi-skew brace.
\end{theorem}
%\comment{what about $p$ and $q$?}\com{What do you mean?} \comment{Are we assuming some conditions on $p$ and $q$?} \com{I think we don't because $\pi_2(G)$ is included in $\aut{\Z_p^2}$ because $\Z_q^\times$ has no elements of order $q$ I would include an explanation before the theorem to emphasize that there are not special conditions.}

\begin{proof}
	The groups 
	\begin{equation}\label{list_sub_abelian_q}
	H_s=\langle \sigma,\ \tau,\ \epsilon \dm_s\rangle\cong \G_s,
	\end{equation}
	for $s\in \B$ are not pairwise conjugated since their image under $\pi_2$ are not and $|\pi_2(H_s)|=q$.

	Let $K=\langle \sigma,\tau\rangle$ and $p:\Hol(A)\longrightarrow \Hol(A)/K\cong \Z_q\times \aut{A}$. Assume that $h\in (1\times \aut{A})\cap H_s$. Then $p(h)\in (1\times\aut{A})\cap \langle \epsilon \dm_s\rangle=1$. Therefore $h\in K\cap (1\times \aut{A})=1$. Hence, according to Lemma \ref{regularity} $H_s$ is regular.

	We show that every regular subgroup $G$ with $|\pi_2(G)|=q$ is conjugate to one of the groups in \eqref{list_sub_abelian_q}. The kernel of $\pi_2|_G$ is the $p$-Sylow subgroup of $A$ and the image of $\pi_2|_G$ is generated by an automorphism of order $q$. Then, according to Remarks \ref{action on q-Syl trivial} and \ref{subgroups of GL}, we can assume that the automorphism is given by $\dm_s$ for $s\in\B$. Therefore
	$$G=\langle \sigma,\ \tau,\ \epsilon^a\dm_s\rangle$$
	for some $a\neq 0$ since $G$ is regular. Then $G$ is conjugate to $H_s$ by the automorphism of $\Z_q$ mapping $\epsilon$ to $\epsilon^{a^{-1}}$.

	Let $B_s$ be the left brace associated to $H_s$. Then $\sigma,\tau\in \ker{\lambda}$ and $\epsilon\in \Fix(B_s)$, so we can apply Lemma \ref{ker fix}(ii), i.e. 
	\begin{align*}
		\sigma^n\tau^m\epsilon^l\circ \sigma^x\tau^y\epsilon^z&=\sigma^n\tau^m\epsilon^l\lambda_{\sigma^n\tau^m\epsilon^l}(\sigma^x\tau^y\epsilon^z)\\
		&=\sigma^n\tau^m\epsilon^l\lambda_\epsilon^l(\sigma^x\tau^y\epsilon^z)\\
&=\sigma^n\tau^m\lambda_\epsilon^l(\sigma^x\tau^y)\epsilon^{l+z} = \sigma^n\tau^m\dm_s^l(\sigma^x\tau^y)\epsilon^{l+z}
	\end{align*}
	 for every $0\leq n,m,x,y\leq p-1,\, 0\leq l,z\leq q-1$. Then \eqref{formula1} follows and the left braces $B_s$ are bi-skew braces according to \cite[Proposition 1.1]{skew_pq}, since $(A,+)=\Z_p^2\times \Z_q$ and $(A,\circ)=\Z_p^2\rtimes_{\dm_s} \Z_q$.
	% $$\underbrace{\tau\circ \tau\circ \ldots\circ \tau}_{x_3}=\tau^{x_3}$$
	% and so $\lambda_{v_1^{x_1} v_2^{x_2} \tau^{x_3}}=\lambda_{v_1^{x_1} \circ v_2^{x_2}\circ \tau^{x_3}}=\lambda_{\tau^{x_3}}=\dm_k^{x_3}$ and then \eqref{formula1} follows.
	% \begin{eqnarray}
	% \begin{pmatrix} n \\ m\\r \end{pmatrix} + \begin{pmatrix} s \\ t\\l \end{pmatrix} =\begin{pmatrix} n+s \\ m+t\\r+l \end{pmatrix},\qquad 
	% \begin{pmatrix} n \\ m \\r\end{pmatrix} \circ \begin{pmatrix} s \\ t \\l \end{pmatrix} =\begin{pmatrix} n+\lambda^r s \\m+\left(\lambda^k\right)^r t \\ r+l\end{pmatrix}\end{eqnarray}
	% for every $0\leq n,m,s,t\leq p-1,\, 0\leq r,l\leq q-1$.
\end{proof}

\begin{proposition}\label{ref for reg}
	%Let $p,q$ be primes such that $p= 1\pmod{q}$. 
	Let $q>2$. The unique left brace of $A$-type with $|\ker{\lambda}|=p$ is $(B,+,\circ)$ where
	\begin{eqnarray}\label{formula_sub_abelian_pq}
	\begin{pmatrix} x_1 \\ x_2\\x_3 \end{pmatrix} + \begin{pmatrix} y_1 \\ y_2\\y_3 \end{pmatrix} =\begin{pmatrix} x_1+y_1 \\ x_2+y_2\\ x_3+y_3 \end{pmatrix},\quad \begin{pmatrix} x_1 \\ x_2 \\x_3\end{pmatrix} \circ \begin{pmatrix} y_1 \\ y_2 \\y_3 \end{pmatrix} 
	=\begin{pmatrix} x_1+g^{x_3}y_1+g^{2^{-1} x_3}x_2 y_2\\
	x_2+g^{2^{-1}x_3}y_2 \\ x_3+y_3\end{pmatrix}
	% \begin{pmatrix} x_1 \\ x_2 \\x_3\end{pmatrix} \circ \begin{pmatrix} y_1 \\ y_2 \\y_3 \end{pmatrix} 
	% &=&\begin{pmatrix} x_1 \\x_2 \\ x_3\end{pmatrix}+\begin{bmatrix}{}
	% 1 & 1& 0\\
	% 0 & 1&0\\
	% 0 & 0 &1
	% \end{bmatrix}^{x_2}\begin{bmatrix}
	% \lambda& 0 & 0\\
	% 0 & \lambda^{\frac{q+1}{2}} & 0\\
	% 0 & 0 &1\end{bmatrix}^{x_3}\begin{pmatrix} y_1 \\y_2 \\ y_3\end{pmatrix}
	\end{eqnarray}
	for every $0\leq x_1,x_2,y_1,y_2\leq p-1,\, 0\leq x_3,y_3\leq q-1$. In particular, $(B,\circ)\cong\G_{2}$.
\end{proposition}

% \begin{lemma}\label{sub_abelian_pq}
% Let $\lambda$ be a fixed element of order $q$ in $\mathbb{Z}_p$. There exists a unique conjugacy class of regular subgroups $G$ of $\Hol(A)$ with $|\pi_2(G)|=pq$. A representive is
% \begin{equation*}\label{list_sub_abelian_pp}
%     H=\langle \sigma, \tau C, \epsilon\dm_{2^{-1}} \rangle\cong \G_{2^{-1}}.
% \end{equation*}
% %where $C(\sigma)=\sigma$, $C(\tau)=\sigma\tau$, $C(\epsilon)=\epsilon$ and $\dm_{\frac{q-1}{2}}$ is the diagonal matrix with entries $\lambda,\lambda^{\frac{q-1}{2}}$.
% \end{lemma}{}

\begin{proof}
	The image under $\pi_2$ of the group 
	$$    H=\langle \sigma,\ \tau C,\ \epsilon\dm_{2^{-1}} \rangle\cong \G_{2^{-1}}\cong \G_2$$
	has size $pq$. Assume that $h=\sigma^n(\tau C)^m (\epsilon\dm_{2^{-1}})^l=\sigma^{n+\frac{m(m-1)}{2}}\tau^m C^m \epsilon^l\dm_{2^{-1}}^l$ is in the stabilizer of $1$, i.e. $\sigma^{n+\frac{m(m-1)}{2}}\tau^m \epsilon^l=1$. Therefore $n=m=l=1$, so $h=1$ and then, since $|H|=p^2q$ we have that $H$ is regular.
	%(indeed the size of $\ker{\pi_2}$ is $p$ and the size of $\pi_2(G)$ is $pq$). Moreover 
	% and it is regular since
	% $$H=\setof{\sigma^{n}\left(\tau C\right)^{m}\left(\epsilon \dm_{2^{-1}}\right)^{l}=\sigma^{n+\frac{m(m-1)}{2}}\tau^{m}\epsilon^{l}C^{m}\dm_{2^{-1}}^{l}}{0\leq n,m\leq p-1,\, 0\leq l\leq q-1}.$$ 
	Let $G$ be a regular subgroup of $\Hol(A)$ with $|\pi_2(G)|=pq$. According to Remarks \ref{action on q-Syl trivial} and \ref{subgroups of GL} we can assume that $\pi_2(G)=\langle C,\dm_s\rangle $ for some $s$ or $\pi_2(G)=\langle C,\widetilde{\dm}\rangle $. In the first case, the group $G$ has the following standard presentation
	$$G=\langle v,\ u \epsilon^a C,\  w\epsilon^b \dm_s\rangle,$$
	% Since $p,q$ do not divide $q-1$ the $\pi_2(G)$ acts trivially on the $q$-Sylows of $A$.  Up to conjugation, a subgroup of order $pq$ of $GL_2(p)$ is generated by
	% $$C=\begin{bmatrix} 1& 1\\ 0& 1 \end{bmatrix},\quad D=\begin{bmatrix} x& 0\\ 0& y \end{bmatrix},$$
	% for some $x,y\neq 0.$
	where $v,u,w\in\langle\sigma,\tau\rangle$. Checking condition \nc for $\left(u \epsilon^a C\right)^p \in \ker{\pi_2|_G}$ we have $a=0$ and by condition \NC we have $v\in \langle\sigma\rangle$. Therefore
	$$G=\langle \sigma,\ \tau^a  C,\  \tau^d\epsilon^b \dm_s\rangle,$$
	where $a\neq 0$ since $G$ is regular. Conjugating by the automorphism $a^{-1}I$ of $\mathbb{Z}_p^2$ 
	%have the following form:
	%$$
	%\alpha=\begin{bmatrix} 1 & 1\\
	%0& 1\end{bmatrix},\quad 
	%\gamma=\begin{bmatrix} \nu & b\\
	%0 & \mu\end{bmatrix},$$
	% Hence taking an opportune power of the second generator we can assume that
	% $$\alpha=\begin{bmatrix} 1 & 1\\
	% 0& 1\end{bmatrix}.%,\quad %\beta=\begin{bmatrix} \nu & 0\\
	% %0 & \mu\end{bmatrix}.
	% $$
	%multiplying the third one with the second one to the power $-m$ 
	we can assume that $a=1$. Therefore 
	$$G=\langle \sigma, \tau  C,  \tau^d\epsilon^b \dm_s\rangle=\langle \sigma,\ \tau  C,\  \epsilon^b C^{-d}\dm_s\rangle$$
	and so $b\neq 0$ and conjugating by the automorphism $b^{-1}$ of $\Z_q$ we can assume that also $b=1$. % =\langle \sigma, \tau  C,  \epsilon C^{-c}\dm_s\rangle,$$
	%$$G=\langle \sigma, \tau \alpha, \tau^c\epsilon^d\gamma\rangle
	%=\setof{v_1^{x_1+k\frac{x_2(x_2-1)}{2}}v_2^{x_2}\tau^{tx_3}\alpha^{x_2}\gamma^{x_3}}{0\leq x_1,x_2\leq p-1,\, 0\leq x_3\leq q-1}$$
	%where $d\neq 0$ since $G$ is regular. 
	%where the second equality follows by multiplying the last generator by a suitable power of the second one. 
	Since
	\begin{equation}\label{eq3.9}
(\epsilon C^{-d}\dm_s)\tau C(\epsilon C^{-d}\dm_s)^{-1}=	\epsilon C^{-d}\dm_s(\tau) \dm_s C\dm_s^{-1}C^{d}\epsilon^{-1}=\tau^{g^s}C^{g^{1-s}} \pmod{\langle \sigma\rangle}
	\end{equation}
	and $C^{-d}\dm_s C\dm_s^{-1}C^d=C^{g^{1-s}} $, by condition \nc we have $\tau^{g^s}C^{g^{1-s}}=(\tau C)^{g^{1-s}}=\tau ^{g^{1-s}}C^{g^{1-s}} \pmod{\langle \sigma\rangle}$ and then $1-s=s \pmod q$, i.e. $s=2^{-1}$. Then $G$ is conjugate to $H$ by $C^n$ where $n=\frac{d}{1-g^{2^{-1}}}$.
	Let $B$ be the skew brace associated to the subgroup $H$. Then $\epsilon\in \Fix(B)$ and $\sigma\in \ker{\lambda}$, and so
	\begin{align*}
		\underbrace{\epsilon\circ \epsilon\circ \ldots\circ \epsilon}_{x_3}&=\epsilon^{x_3}, \qquad
		\underbrace{\tau\circ \tau\circ \ldots\circ \tau}_{x_2}=\sigma^{\frac{x_2(x_2-1)}{2}}\tau^{x_2}=\sigma^{\frac{x_2(x_2-1)}{2}}\circ \tau^{x_2}
	\end{align*}{}
	%and so $\lambda_{v_2^{x_2}}=\lambda_{v_2}^{x_2}=\alpha^{x_2}$. 
	Hence $\lambda_{\sigma^{x_1} \tau^{x_2}\epsilon^{x_3}}=\lambda_{\sigma^{x_1}\circ  \tau^{x_2}\circ  \epsilon^{x_3}}=\lambda_{\tau^{x_2}}\lambda_{\epsilon^{x_3}}=C^{x_2}\dm_{2^{-1}}^{x_3}$. Then \eqref{formula_sub_abelian_pq} follows.
	If $\pi_2(G)=\widetilde{H}$ then we can argue similarly checking the \nc and \NC conditions. From \eqref{eq3.9}, in which we replace $\dm_s$ by $\widetilde{\dm}$, we get that $g^2=1$, hence a contradiction.
\end{proof}{}

For the next proposition, recall that $g$ is an element of order $q$ in $\Z_p^\times$, so for the case $q=2$ we assume $g=-1$.

\begin{proposition}\label{ref for reg2}
	%Let $p,q$ be primes such that $p= 1\pmod{q}$. 
	Let $q=2$. The unique left brace of $\Z_p^2\times \Z_2$-type with $|\ker{\lambda}|=p$ is $(B,+,\circ)$ where
	\begin{eqnarray}\label{formula_sub_abelian_pq2}
	\begin{pmatrix} x_1 \\ x_2\\x_3 \end{pmatrix} + \begin{pmatrix} y_1 \\ y_2\\y_3 \end{pmatrix} =\begin{pmatrix} x_1+y_1 \\ x_2+y_2\\ x_3+y_3 \end{pmatrix},\quad \begin{pmatrix} x_1 \\ x_2 \\x_3\end{pmatrix} \circ \begin{pmatrix} y_1 \\ y_2 \\y_3 \end{pmatrix} 
	=\begin{pmatrix} x_1+y_1+(-1)^{ x_3}x_2 y_2\\
	x_2+(-1)^{x_3}y_2 \\ x_3+y_3\end{pmatrix}
	\end{eqnarray}
	for every $0\leq x_1,x_2,y_1,y_2\leq p-1,\, 0\leq x_3,y_3\leq 1$. In particular, $(B,\circ)\cong\G_{0}$.
\end{proposition}

\begin{proof}
	We can use the very same argument of Proposition \ref{ref for reg}. Let $G$ be a regular subgroup of $\Hol(A)$ such that $\pi_2(G)=\langle C, \dm_s\rangle$. Then, as in Proposition \ref{ref for reg} we have  
	$$G=\langle \sigma,\ \tau C,\ \tau^d\epsilon\dm_s\rangle=\langle \sigma,\ \tau C,\ \epsilon C^{-d}\dm_s\rangle$$ and so
	\begin{equation*}
(\epsilon C^{-d}\dm_s) \tau C (\epsilon C^{-d}\dm_s)^{-1}=		\epsilon C^{-d}\dm_s(\tau) \dm_s C\dm_s^{-1}C^{d}\epsilon^{-1}=\tau^{(-1)^s}C^{(-1)^{1-s}} \pmod{\langle \sigma\rangle}.
	\end{equation*}
	By condition \nc, we have $\tau^{(-1)^s}C^{(-1)^{1-s}}=(\tau C)^{(-1)^{1-s}}=\tau ^{(-1)^{1-s}}C^{(-1)^{1-s}}\pmod{\langle \sigma\rangle}$. Then $(-1)^{1-s}=(-1)^s$, and so it follows that $-1=1$, hence $p=2$, contradiction. Therefore, necessarily $\pi_2(G)=\widetilde{H}$. Using condition \nc and \NC we have that $G$ is conjugate to the following group
	$$K=\langle \sigma,\ \tau C,\ \epsilon \widetilde{D} \rangle.$$
	Using the same idea as in Proposition \ref{ref for reg} we can show that the skew brace associated to $G$ is given by \eqref{formula_sub_abelian_pq2}.
\end{proof}

The following table summarizes the results of this subsection.

% \begin{table}[H]
% \centering
%     \begin{tabular}{c|c|c|c|c|c|c}
%          $|\ker{\lambda}|$ &  $\mathbb{Z}_{p}^2\times\Z_q$  &  $\G_0$ & $\G_{-1}$ & $\G_{1}$ & $\G_{2}$ & $\G_k$,\, $k\neq 0,\pm 1, 2$  \\
%     \hline
%     $p$  & - &- &- &- & 1 &-  \\
%     $pq$ & 1 &- &- &- &- &- \\
%     $p^2$ &- & 1 & 1 & 1 & 1 & 1 \\
%     $p^2q$ & 1 &- &- &- &- &-
%          \end{tabular}
%     \caption{Enumeration of left braces of $\Z_p^2\times\Z_q$-type for $p=1\pmod{q}$. Note that, for $q=3$ we have that $2=-1 \pmod{3}$.}

%     \label{table:abelian_p=1(q)}
% \end{table}

% \com{I prefer the following table, it is more compact.}

\begin{table}[H]
	\centering
	\begin{tabular}{c|c|c|c}
		$|\ker{\lambda}|$ &  $\mathbb{Z}_{p}^2\times\Z_q$  &   $\G_{2}$ & $\G_k$,\, $k\in \B\setminus\{2\}$  \\
		\hline
		$p$  & - & 1 &-  \\
		$pq$ & 1 &- &- \\
		$p^2$ &- & 1 & 1 \\
		$p^2q$ & 1 &- &-
	\end{tabular}
	\caption{Enumeration of left braces of $\Z_p^2\times\Z_q$-type for $p=1\pmod{q}$. Note that, for $q=3$ we have that $2=-1 \pmod{3}$. For the case $q=2$, $\G_2=\G_0$ and the last column correspond to $\G_1$.}
	
	\label{table:abelian_p=1(q)}
\end{table}

\section{Left braces of size $p^2q$ with $p=-1 \pmod{q}$}\label{section:p=-1(q)}

%  \begin{table}[H]
%  \centering
%      \begin{tabular}{c|c|c}
%           $(A,+)$  &  $|\pi_2(G)|$  \\
%      \hline
%     %  $\mathbb{Z}_{p^2q}$  & $\com{p}$\\
%     % $\mathbb{Z}_{p}^2\times \mathbb{Z}_q$ & $\com{p,q}$\\
%  $\G_{F}$  & $\com{q, p^2},p^2q$\\

%           \end{tabular}
% % %    \caption{$p=-1\pmod{q}$: implemented on GAP [ [ 5, 3 ], [ 11, 3 ] ]}
%  \end{table}

%\subsection{Groups and autormorphisms}

In this section we assume that $p$ and $q$ are odd primes such that $p=-1\pmod{q}$ and that $\h=x^2+\k x+1$ is an irreducible polynomial over $\Z_p$ such that its companion matrix 
$$F=\begin{bmatrix}
0 & -1 \\
1 & -\xi
\end{bmatrix}$$ 
has order $q$.
% In particular,
% \begin{equation}\label{solutions of h(F)=0} 
% \h(F^n)=0 \, \text{ if and only if } \, n=\pm 1.
% \end{equation}
According to \cite[Proposition 21.17]{EnumerationGroups}, there exists a unique non abelian group of size $p^2q$. Its presentation is given in \cite{auto_pq}:% are the following:
%\begin{lemma}\cite[Proposition 21.17]{EnumerationGroups}
% Let $p,q>2$ be primes such that $p=-1\pmod{q}$, $f=x^2+kx+1$ an irreducible polynomial over $\mathbb{Z}_p$ such that its companion matrix $F$ has order $q$. Then t
%The groups of size $p^2q$ are the following:
%\begin{itemize}
%   \item[(i)] $\mathbb{Z}_{p^2q}$. 
%  \item[(ii)] $\mathbb{Z}_{p}^2\times \mathbb{Z}_q$. 
%    \item[(iii)] 
$$\G_F=\langle \sigma,\tau,\epsilon\,|\, \sigma^p=\tau^p=\epsilon^q=1,\, \epsilon \sigma \epsilon^{-1}=\tau,\, \epsilon \tau \epsilon^{-1}=\sigma^{-1}\tau^{-\k}\rangle \cong \mathbb{Z}_p^2\rtimes_F \mathbb{Z}_q.$$
%\end{itemize}
%\end{lemma}
The enumeration of the left braces of size $p^2q$ with $p=-1\pmod q$ according to their additive and multiplicative group is the following:

\begin{table}[H]
	\centering
	\begin{tabular}{c|c|c|c}
		$+ \backslash \circ$ & $\Z_{p^2q}$ & $\Z_p^2\times\Z_q$ & $\G_F$ \\
		\hline
		$\Z_{p^2q}$ & $2$  &- &-  \\
		$\Z_p^2\times\Z_q$ &- & $2$ & $1$  \\
		%        $\G_F$ &- & $2$  & $p+2q-4$ 
	\end{tabular}
	\caption{Enumeration of left braces of size $p^2q$ with $p=-1\pmod{q}$.}
\end{table}

\subsection{Left braces of cyclic type}\label{subsection:p=-1(q)_cyclic}
In this section we denote by $A$ the cyclic group $\Z_{p^2q}$. In this case, the size of $\pi_2(G)$ for a regular subgroup $G$ of $\Hol(A)$ divides $p$. Since in this case $q\neq 1\pmod{p}$, we can apply Theorem \ref{prop:cyclic_ker=pq} and so the enumeration of left braces of $A$-type is as in the following table.

% \begin{proposition}
% %Let $p,q>2$ be primes with $p= -1\pmod{q}$. 
% The unique skew brace of cyclic type with $|\ker{\lambda}|=pq$ is $(B,+,\circ)$ where 
% \begin{equation}\label{formula cyclic pq}
%     \begin{pmatrix} n \\ m \end{pmatrix} + \begin{pmatrix} s \\ t \end{pmatrix} =\begin{pmatrix}  n+s \\ m+t \end{pmatrix}, \qquad\qquad \begin{pmatrix} n \\ m \end{pmatrix} \circ \begin{pmatrix} s \\ t \end{pmatrix} =\begin{pmatrix}  n+s+pns \\ m+t \end{pmatrix}
%     \end{equation}
%     for every $0\leq n,s\leq p^2-1$ and $0\leq m,t\leq q-1$. In particular, $(B,\circ)\cong \mathbb{Z}_{p^2q}$.
% \end{proposition}

% \begin{table}[H]
% \centering
%     \begin{tabular}{c|c|c|c}
%          $|\ker{\lambda}|$ &  $\Z_{p^2q}$ & $\Z_p^2\times\Z_q$ &  $\G_F$ \\
%     \hline
%     $pq$ & $1$ &- &- \\
%     $p^2q$ & $1$ &- &-
%     \end{tabular}
%     \caption{Number of skew braces of $\Z_{p^2q}$-type of size $p^2q$ for $p=-1\pmod{q}$.}
% \end{table}

\begin{table}[H]
	\centering
	\begin{tabular}{c|c}
		$|\ker{\lambda}|$ &  $\Z_{p^2q}$ \\
		\hline
		$pq$ & $1$  \\
		$p^2q$ & $1$
	\end{tabular}
	\caption{Number of left braces of $\Z_{p^2q}$-type of size $p^2q$ for $p=-1\pmod{q}$.}
\end{table}

\subsection{Left braces of $\mathbb{Z}_p^2\times \mathbb{Z}_q$-type}\label{subsection:p=-1(q)_abelian_non-cyclic}

Let $A=\mathbb{Z}_p^2\times \mathbb{Z}_q$ and let $G$ be a regular subgroup of $\Hol(A)$. Then the size of $\pi_2(G)$ divides $p^2q$ and $|\aut{A}|$, so divides $pq$ since $p=-1\pmod{q}$.

\begin{proposition}
	Let $G$ be a regular subgroup of $\Hol(A)$. Then $|\pi_2(G)|\neq pq$. 
\end{proposition}
\begin{proof}
	Let $G$ be such a group, then $\ker{\pi_2}$ is a normal subgroup of size $p$ of $G$. Therefore $G$ is not isomorphic to $\G_F$, since it has no normal subgroup of size $p$. On the other hand, $G$ is not abelian since $\aut{A}$ have no abelian subgroups of order $pq$.
	%Let $G\cong \G_F$. The elements of order $q$ of $G$ act irreducibly on the $p$-Sylow of $G$, and therefore $G$ has no normal subgroup of order $p$ and then $\ker{\pi_2}$ is not normal in $G$, contradiction. 
\end{proof}

\begin{lemma}
	%Let $p,q$ be primes such that $p= -1\pmod{q}$. 
	There exists a unique left brace of $A$-type with $|\ker{\lambda}|=p^2$. It is given by $(B,+,\circ)$ where
	\begin{eqnarray}\label{formula_sub_abelian_q_p=-1}
	\begin{pmatrix} x_1 \\ x_2\\x_3 \end{pmatrix} + \begin{pmatrix} y_1 \\ y_2\\y_3 \end{pmatrix} =\begin{pmatrix} x_1+y_1 \\ x_2+y_2\\ x_3+y_3 \end{pmatrix},\qquad 
	\begin{pmatrix} x_1 \\ x_2 \\x_3\end{pmatrix} \circ \begin{pmatrix} y_1 \\ y_2 \\y_3 \end{pmatrix} 
=\begin{pmatrix} x_1 \\ x_2 \\x_3\end{pmatrix} \circ \begin{pmatrix} y_1 \\ y_2 \\y_3 \end{pmatrix} =\begin{pmatrix}  \begin{pmatrix}
x_1 \\x_2
\end{pmatrix} + F^{x_3}\begin{pmatrix} y_1  \\ y_2\end{pmatrix} \\ x_3 + y_3 \end{pmatrix}
%
%	=\begin{pmatrix} x_1 \\
%	x_2 \\ x_3\end{pmatrix}+\begin{bmatrix} 0 & -1 & 0\\
%	1& -\k & 0\\
%	0& 0 & 1\end{bmatrix}^{x_3}\begin{pmatrix} y_1  \\ y_2 \\ y_3\end{pmatrix},
	\end{eqnarray}
	for every $0\leq x_1,x_2,y_1,y_2\leq p-1,\, 0\leq x_3,y_3\leq q-1$. In particular, $(B,\circ)\cong \G_F$ and $B$ is a bi-skew brace. %Here the notation $[*]^{x}$ stands for the power of the matrix $[*]$.
	%and $x^2+kx+1$ is a fixed irreducible polynomial over $\mathbb{Z}_p$ such that its companion matrix has order $q$.   
\end{lemma}

%\com{Referee comment: It is better to explain what is $[*]^{x_3}$}\comment{Shall we put the explanation of the notation somewhere else? Maybe right before the Lemma. Are we using power notation for matrices elsewhere?}\com{I think we don't use it anywhere else} \comment{I see, still don't like the statement very much. Basically the first 2 components are like $(x_1,x_2)+F^{x_3}(y_1,y_2)$ and the last is $x_3+y_3$. Is there a nice way to write it in this way? So we can avoid to explain the notation (by the way we do not explain what the matrix is in details)} \com{I was thinking to write the matrix by blocks but I think it is even worse. What do you think about this?:} 

%$\begin{pmatrix} x_1 \\ x_2 \\x_3\end{pmatrix} \circ \begin{pmatrix} y_1 \\ y_2 \\y_3 \end{pmatrix} =\begin{pmatrix}  \begin{pmatrix}
%x_1 \\x_2
%\end{pmatrix} + F^{x_3}\begin{pmatrix} y_1  \\ y_2\end{pmatrix} \\ x_3 + y_3 \end{pmatrix}$

%\com{I don't like the parentheses in the first coordinate but it is standard to think tuples as tuples of tuples, do you like it? }

% \begin{lemma}\label{sub_abelian_q2}
% Let $p,q$ be primes such that $p=-1\pmod q$. There exists a unique conjugacy class of regular subgroups $G$ of $\Hol(A)$ with $|\pi_2(G)|=q$. A representative is
% $$H=\langle \sigma,\tau, \epsilon F\rangle\cong \G_F$$
% where $F$ is the irreducible automorphism of $\mathbb{Z}_p^2$ of order $q$ given by the companion matrix of the irreducible polynomial $x^2+kx+1$. 
% \end{lemma}{}
\begin{proof}
	Let $G$ be a regular subgroup of $\Hol(A)$ with $|\pi_2(G)|=q$. Any automorphism of order $q$ of $A$ acts trivially on the $q$-Sylow subgroup of $A$. Up to conjugation, we can assume that $\pi_2(G)$ is generated by $F$. %, the companion matrix of the irreducible polynomial $x^2+kx+1$. 
	Therefore,
	$$G=\langle \sigma,\ \tau,\ \sigma^n \tau^m\epsilon^a F\rangle=\langle \sigma,\ \tau,\ \epsilon^a F\rangle
	%=\langle v_1,v_2,\tau^k\beta\rangle=\setof{v_1^{x_1}v_2^{x_2}\tau^{k x_3}\beta^{x_3}}{0\leq x_1,x_2\leq p-1,\, 0\leq x_3\leq q-1}
	$$
	where $a\neq 0$ since $G$ is regular. 
	%The subgroups generated by the irreducible automorphism of order $q$ are all conjugate i.e. there exist some $f\in GL_2(p)$ and some $1\leq l\leq q-1$ such that $f\beta f^{-1}=\alpha^l$. Then
	%$$f G f^{-1}=\langle v_1,v_2, \tau^{k} \alpha^l\rangle=\langle v_1,v_2, \tau^{k l^{-1}} \alpha\rangle.$$
	We can conjugate $G$ by the automorphism taking $\epsilon^{a}$ to $\epsilon$ and fixing $\sigma$ and $\tau$, so we can assume that $a=1$. It is straightforward to verify that such group is regular. Let $B$ be the left brace associated to $G$. We have $\sigma,\tau\in \ker{\lambda}$ and $\epsilon\in \Fix(B)$ and then $B= \ker{\lambda}+\Fix(B)$. So, formula \eqref{formula_sub_abelian_q_p=-1} follows by Lemma \ref{ker fix}. According to \cite[Corollary 1.2]{skew_pq}, we have that \eqref{formula_sub_abelian_q_p=-1} defines a bi-skew brace.
\end{proof}

According to Theorem \ref{sub_abelian_p} there exists a unique non trivial left brace of size $p^2q$ with the $|\ker{\lambda}|=pq$. Hence we have the following enumeration:

\begin{table}[H]
	\centering
	\begin{tabular}{c|c|c}
		$|\ker{\lambda}|$ &   $\Z_p^2\times\Z_q$ &  $\G_F$ \\
		\hline
		$p^2$ &- & $1$ \\
		$pq$  & $1$ &- \\
		$p^2q$  & $1$ &-
	\end{tabular}
	\caption{Enumeration of left braces of $\Z_p^2\times\Z_q$-type for $p=-1\pmod{q}$.}
\end{table}

\section{Left braces of size $p^2q$ with $q=1 \pmod{p}$ and $q\neq 1 \pmod{p^2}$}\label{section:q=1(p)}

In this section we will assume that $q=1\pmod{p}$ but $q\neq 1\pmod{p^2}$. Also we assume $p>2$, and we deal with the case $p=2$ in Section \ref{4p (i)}. If we denote by $\r$ a fixed element of order $p$ in $\Z_q^\times$, then the non-abelian groups relevant for this section are the following, \cite[Proposition 21.17]{EnumerationGroups}:
%\subsection{Groups and automorphisms}

%\begin{lemma}\cite[Proposition 21.17]{EnumerationGroups}
%   The groups of size $pq$ are the following:
% Let $q=1 \pmod{p}$, $\r$ be an element of order $p$ in $\Z_q^\times$. Then the groups of size $p^2q$ are the following:
\begin{itemize}
	%    \item[(i)] $\mathbb{Z}_{p^2q}=\langle \sigma,\tau\ |\ \sigma^{p^2}=\tau^q=1,\ \tau\sigma=\sigma\tau\rangle$.
	%   \item[(ii)] $\mathbb{Z}_{p}^2\times \mathbb{Z}_q=\langle \sigma, \tau, \epsilon\ |\ \sigma^p=\tau^p=\epsilon^q=1,\ [\sigma,\tau]=[\sigma,\epsilon]=[\tau,\epsilon]=1\rangle$.
	\item[(i)] $\mathbb{Z}_{p}\times (\mathbb{Z}_{q}\rtimes_{\r} \mathbb{Z}_{p})=\langle \sigma,\tau,\epsilon\,|\, \sigma^p=\tau^p=\epsilon^q=1,\ [\epsilon,\tau]=[\tau, \sigma]=1,\, \sigma\epsilon\sigma^{-1}=\epsilon^{\r} \rangle$.
	\item[(ii)] $\mathbb{Z}_{q}\rtimes_{\r} \mathbb{Z}_{p^2}=\langle \sigma,\tau\, | \, \sigma^{p^2}=\tau^{q}=1,\, \sigma\tau\sigma^{-1}=\tau^{\r} \rangle$.
	% \com{I interchanged $\sigma$ and $\tau$ so that $\sigma$ has $p$-order and $\tau$ has $q$-order}
\end{itemize}{}

% \com{This case has to be separated because in this section $q\neq 1\pmod{p^2}$}

% If $q=1 \pmod{p^2}$ and $h$ is an element of multiplicative order $p^2$ in $\mathbb{Z}_q$ then we have also the group
% \begin{itemize}
%  \item[(v)] $G=\langle \sigma,\tau\, | \, \sigma^q=\tau^{p^2}=1,\, \tau\sigma\tau^{-1}=\sigma^h \rangle\cong \mathbb{Z}_{q}\rtimes_{h} \mathbb{Z}_{p^2}$.
% \end{itemize}{}

%\end{lemma}{}

We summarize in the following table the enumeration of skew braces according to the additive and multiplicative isomorphism class of groups.

\begin{table}[H]
	\centering
	\begin{tabular}{c|c|c|c|c}
		$+ \backslash \circ$ & $\Z_{p^2q}$ & $\Z_q\rtimes_{\r}\Z_{p^2}$  & $\Z_p^2\times\Z_q$ & $\Z_p\times(\Z_q\rtimes_{\r}\Z_p)$ \\
		\hline
		$\Z_{p^2q}$ & $2$ & $p$ &- & -  \\
		$\Z_p^2\times\Z_q$ &- &- & $2$ & $4$   \\
		%        $\Z_p\times(\Z_q\rtimes_{\r}\Z_p)$ &- & $4$ & $6p-4$ &-  \\
		%       $\Z_q\rtimes_{\r}\Z_{p^2}$ & $2p$ &- &- & $2p(p-1)$
	\end{tabular}
	\caption{Enumeration of left braces of size $p^2q$ with $q=1\pmod{p}$ and $q\neq 1\pmod{p^2}$.}
\end{table}

\subsection{Left braces of cyclic type}\label{subsection:q=1(p)_cyclic}

Let $A$ denote the cyclic group $\Z_{p^2q}$.
% and by $\r$ a fixed element of order $p$ in $\Z_q^\times$. 
%For the automorphism group we employ the same notation as in Subsection \ref{subsection:p=1(q)_cyclic}.
If $G$ is a regular subgroup of $\Hol(A)$ then $|\pi_2(G)|\in\{1,p,p^2\}$. %First we show that there are not regular subgroups in the case $p^2$.

The group $\aut{A}$ has a unique $p$-Sylow subgroup, which is generated by $\varphi_{p+1,1}$ and $\varphi_{1,\r}$, which is elementary abelian. So the subgroups of order $p$ in $\aut{A}$ are 
\begin{equation}\label{lem:aut_cyclic_subgroups_p}
\langle \varphi_{p+1,1} \rangle \qquad\text{and}\qquad \langle \varphi_{jp+1,\r}\rangle
\end{equation}
where $0\leq j\leq p-1$.

\begin{proposition}\label{prop:cyclic_ker=q_non-existence}
	% Let $p,q$ be primes such that $q=1\pmod{p}$, $q\neq 1\pmod{p^2}$ and 
	Let $G$ be a regular subgroup of $\Hol(A)$. Then $|\pi_2(G)|\neq p^2$.
	\begin{proof}
		Let $G$ be such group. The image by $\pi_2$ of $G$ is the unique $p$-Sylow subgroup $\langle \varphi_{p+1,1},\ \varphi_{1,\r} \rangle$ of $\aut A$. The unique subgroup of $A$ of size $q$ is $\langle\tau\rangle$, therefore $G$ has the following standard presentation:
		\[
		%  G=
		%  \langle \tau,\ \sigma^a\tau^c\varphi_{p+1,1},\ \sigma^b\tau^d\varphi_{1,\r} \rangle 
		G= \langle \tau,\ \sigma^a\varphi_{p+1,1},\ \sigma^b\varphi_{1,\r} \rangle.
		\]
		The group $\pi_2(G)$ is elementary abelian and so, by the \nc conditions we have $a,b=0\pmod p$. By regularity, $a,b\neq 0\pmod{p^2}$. Let $a=pa'$ and $b=pb'$ for some $1\leq a',b'\leq p-1$. Hence $(\sigma^a \varphi_{p+1,1})^{-1}(\sigma^b \varphi_{1,r})^{\frac{a'}{b'}}=\varphi_{p+1,1}^{-1}\varphi_{1,r}^{\frac{a'}{b'}}\in G$ and so $G$ is not regular, contradiction.
		
		%Hence, $\pi_1(G)\subseteq \langle \sigma^p,\tau\rangle$ and so $G$ is not regular, contradiction. \comment{Expand?}\com{You mean the last sentence about $\pi_1$?} \comment{Yeah, in the other paper we have some lemma for this. Alternatively we can say: Let $a=pa'$ and $b=pb'$ for some $0\leq a',b'\leq p-1$. Hence $(\sigma^a \varphi_{p+1,1})^{-b'}(\sigma^b \varphi_{1,r})^{a'}=\varphi_{p+1,1}^{-b'}\varphi_{1,r}^{a'}\in G$ and so $G$ is not regular.}
	\end{proof}
	
\end{proposition}

\begin{proposition}\label{prop:cyclic_ker=pq_2}
	The left braces of cyclic type of size $p^2q$ with $|\ker{\lambda}|=pq$ are $B_{(j,k)}=(A,+,\circ)$ for $(j,k)\in \{(1,0)\}\cup \setof{(j,1)}{0\leq j\leq p-1}$ where
	\begin{equation*}\begin{pmatrix} n \\ m \end{pmatrix} + \begin{pmatrix} s \\ t \end{pmatrix} =\begin{pmatrix}  n+s \\ m+t \end{pmatrix}, \qquad\qquad \begin{pmatrix} n \\ m \end{pmatrix} \circ \begin{pmatrix} s \\ t \end{pmatrix} =\begin{pmatrix}  n+(jnp+1)s \\ m+\r^{kn}t \end{pmatrix}
	\end{equation*}
	for every $0\leq n,s\leq p^2-1$ and $0\leq m,t\leq q-1$. In particular, 
	$$(B_{(j,k)},\circ)\cong \begin{cases}\mathbb{Z}_{p^2q}\, \text{ if } (j,k)=(1,0),\\
	\mathbb{Z}_{q}\rtimes_{\r} \mathbb{Z}_{p^2}, \, \text{otherwise.}\end{cases}$$
\end{proposition}

\begin{proof}
	The groups
	\begin{equation*}
		H=\langle \sigma^p,\ \tau,\ \sigma\varphi_{p+1,1}\rangle \cong \Z_{p^2q}, \quad G_{j}=\langle \sigma^p,\ \tau,\ \sigma\varphi_{jp+1,r}\rangle \cong \mathbb{Z}_{q}\rtimes_{\r} \mathbb{Z}_{p^2},
	\end{equation*}
	for $0\leq j\leq p-1$ are regular subgroups and they are not conjugate, since their image under $\pi_2$ are not.

	Let $G$ be a regular subgroup with $|\pi_2(G)|=p$, then $\pi_2(G)$ is one of the groups in \eqref{lem:aut_cyclic_subgroups_p}.
	If $\pi_2(G)=\langle \varphi_{p+1,1}\rangle$ we can argue as in the proof of Theorem \ref{prop:cyclic_ker=pq} and the associate brace is $B_{1,0}$.

	Assume that $\pi_2(G)=\langle \varphi_{jp+1,\r}\rangle$ for some $0\leq j\leq p-1$. Then the kernel of $\pi_2$ is the unique subgroup of order $pq$, namely $\langle\sigma^p,\tau\rangle$ and so $G$ has the following standard presentation:
	\[
	G=\langle \sigma^p,\ \tau,\ \sigma^a\varphi_{jp+1,\r}\rangle,
	\]
	where $1\leq a\leq p-1$ since $G$ is regular. Then, $G$ is conjugate to $G_j$ by $\varphi_{a^{-1},1}$.
	%we can assume that $a=1$ by conjugating by the automorphism $\varphi_{a^{-1},1}$. 
	Let $B_{j,1}$ be the left brace associate to $G_j$. Then, since $\langle\tau,\sigma^p\rangle\leq \ker{\pi_2}$ and $\underbrace{\sigma\circ\sigma\circ\ldots\circ\sigma}_n=\sigma^{\frac{n(n-1)p}{2}+n}=\sigma^{\frac{n(n-1)p}{2}}\circ \sigma^n$ we have that 
	$$\lambda_{\sigma^n\tau^m}=\lambda_{\tau^m\circ \sigma^n}=\lambda_{\sigma^n}=\lambda_{\sigma}^n=\varphi_{jp+1,\r}^n $$
	and so the formula in the statement follows.
\end{proof}

We summarize the results of this subsection in the following table:

\begin{table}[H]
	\centering
	\begin{tabular}{c|c|c}
		$\ker{\lambda}$ &  $\mathbb{Z}_{p^2q}$  &  $\Z_q\rtimes_{\r}\Z_{p^2}$  \\
		\hline
		$pq$ & $1$ & $p$ \\
		$p^2q$ & $1$ &-
	\end{tabular}
	\caption{Enumeration of left braces of cyclic type of size $p^2q$ for $q=1\pmod{p}$ and $q\neq 1\pmod{p^2}$.}
\end{table}

\subsection{Left braces of $\mathbb{Z}_p^2\times \mathbb{Z}_q$-type}\label{subsection:q=1(p)_abelian_non-cyclic}

In this section we denote by $A$ the group $\mathbb{Z}_p^2\times \mathbb{Z}_q$ and by $C$ the matrix defined in Remark \ref{subgroups of GL}.

\begin{proposition}\label{sub_abelian_p2}
	The left braces of $A$-type with $|\ker{\lambda}|=pq$ are $(B_{i,j},+,\circ)$ where $0\leq i,j\leq 1$ and $(i,j)\neq (0,0)$, where
	% \begin{align*}
	\begin{equation*}\begin{pmatrix} n \\ m \\ l \end{pmatrix} + \begin{pmatrix} s \\ t \\ u\end{pmatrix} =\begin{pmatrix}  n+s \\ m+t \\ l+u \end{pmatrix}, \qquad\qquad \begin{pmatrix} n \\ m \\ l\end{pmatrix} \circ \begin{pmatrix} s \\ t \\ u\end{pmatrix} =\begin{pmatrix}  n+s+jmt \\ m+t \\ l+\r^{im}u  \end{pmatrix},
	\end{equation*}
	for every $0\leq n,m,s,t\leq p-1$ and $0\leq l,u\leq q-1$. In particular, 
	$$(B_{i,j},\circ)\cong \begin{cases} \mathbb{Z}_p^2\times \mathbb{Z}_q,\, \text{ if } (i,j)=(0,1),\\
	\mathbb{Z}_p\times (\mathbb{Z}_q\rtimes_{\r} \mathbb{Z}_p),\, \text{otherwise.}\end{cases}$$ 
\end{proposition}

\begin{proof}
	Let $G\leq \Hol(A)$ be a regular subgroup with $|\pi_2(G)|=p$. Then, up to conjugation, the image $\pi_2(G)$ is generated by $\r^i C^j$, for $i,j\in \{0,1\}$ and $(i,j)\neq (0,0)$. The kernel of $\pi_2|_G$ has order $pq$, so $\ker{\pi_2|_G}=\langle \epsilon, v\rangle$ for some $v\in \langle \sigma,\tau\rangle$. Therefore we have the following standard presentation:
	$$G_{i,j}=\langle v,\ \epsilon,\ u\r^i C^j\rangle,$$
	for some $u\in \langle \sigma,\tau\rangle$. If $(i,j)=(0,1)$ we can argue as in the proof of Theorem \ref{sub_abelian_p} and then the formula for this case follows.
	
	Assume that $(i,j)=(1,0)$. Then, up to conjugation by an element of $GL_2(p)$ we can assume that $v=\sigma$ and $u=\tau$.
	
	Consider now the case $(i,j)=(1,1)$. By condition \NC we have $v\in \langle \sigma\rangle$ and then 
	$$G_{1,1}=\langle \sigma,\ \epsilon,\ \tau^a\r C\rangle,$$
	for $a\neq 0$. We can assume $a=1$, otherwise we conjugate by $a^{-1}I$.
	
	Let $B_{i,j}$ be the skew brace associated to $G_{i,j}$. In the last two cases, we have that $\lambda_{\sigma^n\tau^m\epsilon^l}=\lambda_{\tau^m}=\lambda_\tau^m$, since $\sigma,\epsilon\in \ker{\pi_2}$ and $\lambda_\tau(\tau)=\tau\pmod {\langle\sigma\rangle}$. Then the claim follows.

\end{proof}{}

\begin{proposition}\label{sub_abelian_pp2} 
	Let $w$ a fixed quadratic non residue modulo $p$. 
	The left braces of $A$-type with $|\ker{\lambda}|=q$ are $(B_s,+,\circ)$ for $s\in \{1,w\}$, where
	\[
	\begin{pmatrix}
	n \\ m \\ l
	\end{pmatrix} + \begin{pmatrix}
	x \\ y \\ z
	\end{pmatrix} = \begin{pmatrix}
	n+x \\ m+y \\ l+z
	\end{pmatrix}, \quad
	\begin{pmatrix}
	n \\ m \\ l
	\end{pmatrix} \circ
	\begin{pmatrix}
	x \\ y \\ z
	\end{pmatrix} =
	\begin{pmatrix}
	n+x+yms^{-1} \\ m+y \\ l+ z \r^{n-m\frac{m-s}{2s}}
	\end{pmatrix},
	\]
	for every $0\leq n,m,x,y\leq p-1$ and $0\leq l,z\leq q-1$. In particular, $(B_s,\circ)\cong \Z_p \times (\Z_q\rtimes_{\r} \Z_p)$.
\end{proposition}

%\comment{This is quite a mess. We should try to improve it}\com{The formula was wrong, now it is shorter. I improved a few things in the proofs, we don't need the automorphism $h'$ and I also changed notation at the end. Now it should be OK.}

\begin{proof}
	The groups 
	$$G_s=\langle \epsilon,\ \tau^s C,\ \sigma \r \rangle\cong \Z_p \times (\Z_q\rtimes_{\r} \Z_p)$$
	for $s\in \{1,w\}$ are regular subgroups of $\Hol(A)$ with $|\pi_2(G_s)|=p^2$ and they are not conjugated. %The groups $G_s$ are regular and t
	Indeed, if they are conjugate by $h$, then $h\in N_{\aut{A}}(C)$ the normalizer of $C$ in $\aut{A}$, say $$h=\begin{bmatrix} x & y \\ 0 & z\end{bmatrix}r^a.$$  
	So we have that
	$h\tau C h^{-1}  = \sigma^y\tau^z C^{\frac{x}{z}}$ and %, \, % \in \langle \tau^wC H, \sigma r H\rangle,\\
	$h \sigma r h^{-1} =\sigma^x r$ belong to $G_w$.
	Then $x=1$ and therefore $\sigma^y\tau^z C^{\frac{1}{z}}=(\tau^wC)^{\frac{1}{z}}$. Comparing the powers of $\tau$ in both sides, we have $w=z^2$, a contradiction.
	
	Let $G\leq \Hol(A)$ be a regular subgroup and let $|\pi_2(G)|=p^2$. Then, up to conjugation, we have that $\pi_2(G)=\langle C,r\rangle$. Hence the group $G$ has the standard presentation
	$$G=\langle \epsilon,\ u C,\ v \r\rangle,$$
	for some $u,v\in \langle \sigma,\tau\rangle$. By the \nc condition, the last two generators commute modulo $\langle\epsilon\rangle$ and this implies that $v\in \langle \sigma\rangle$. Therefore $G$ has the form
	$$G_u=\langle \epsilon,\ u C,\ \sigma^m \r\rangle$$
	where $m\neq 0$ and so conjugation by $m^{-1}I$ allows us to assume $m=1$.
	
	Let assume that $u=\sigma^s$. Since $G_u$ is regular, $s\neq  0$ and then $(\sigma^s C)^{-s^{-1}}\sigma r =C^{-s^{-1}}r\in G$. Hence, by Lemma \ref{regularity}, $G_u$ is not regular, contradiction.
	
	If $u=\sigma^s\tau^t$ for some $t\neq 0$, we can conjugate by $C^{-\frac{s}{t}}$ to assume that $u=\tau^t$. If $t=z^2$ for some $z$ then $G_{\tau^{z^2}}$ is conjugated to $G_{1}$ by $h$ where
	$$h=\begin{bmatrix} 1 & 2^{-1}(1-z^{-1})\\ 0 & z^{-1}\end{bmatrix}.$$
	Otherwise, we can write $t=w z^2$ for some $z$ where $w$ is a quadratic non residue modulo $p$, then $hG_{\tau^t}h^{-1}=G_{w}$ where $h$ is as above.

	Let $B_s$ be the skew brace associated to $G_s$. Then $\lambda_\sigma|_{\langle\sigma,\tau\rangle}=$id$|_{\langle\sigma,\tau\rangle}$ and so $\lambda_{\sigma^n}=\lambda_\sigma^n$ according to Lemma \ref{ker fix}. Moreover we have that
	\begin{equation}\label{to use here}
	\underbrace{\tau^s\circ \cdots \circ \tau^s}_m=\sigma^{s\frac{m(m-1)}{2}}\tau^{ms}=\sigma^{s\frac{m(m-1)}{2}}\circ \tau^{ms}\quad \Longleftrightarrow \quad \tau^m=\tau^{\frac{m}{s}s}=\left( \sigma^{\frac{m(m-s)}{2s}}\right)^\prime \circ\underbrace{ \tau^s\circ\cdots\circ \tau^s}_{\frac{m}{s}} 
	\end{equation}
	for every $m\in \mathbb{N}$. Hence, using \eqref{to use here} and that $\lambda_\sigma(\tau)=\tau$ it follows that 
	\begin{align*}
		\lambda_{\epsilon^l \sigma^n \tau^m}&=\lambda_{\epsilon^l \circ \sigma^n \circ \tau^m}=\lambda_{\sigma}^n\lambda_{\tau^m}    \\
		&= r^n r^{-\frac{m(m-s)}{2s}} C^{\frac{m}{s}}=r^{n-\frac{m(m-s)}{2s}} C^{\frac{m}{s}}.
	\end{align*}
	Hence, the formula for the operation $\circ$ follows.

\end{proof}

We summarize the content of this subsection in the following table:

\begin{table}[H]
	\centering
	\begin{tabular}{c|c|c}
		$|\ker{\lambda}|$ &  $\mathbb{Z}_p^2\times\Z_q$  &  $\Z_p\times(\Z_q\rtimes_{\r}\Z_p)$  \\
		\hline
		$q$ &- & $2$ \\
		$pq$ & $1$ & $2$ \\
		$p^2q$ & $1$ &-
	\end{tabular}
	\caption{Enumeration of left braces of $\Z_p^2\times\Z_q$-type of size $p^2q$ for $q=1\pmod{p}$ and $q\neq 1\pmod{p^2}$.}\label{table enum abelian not cyclic q=1 (p)}
\end{table}

\section{Left braces of size $p^2q$ with $q=1\pmod{p^2}$}\label{section:q=1(p^2)}

In this section we assume that $q=1\pmod{p^2}$ where $p$ is an odd prime. We will denote by $h$ a fixed element of order $p^2$ in $\Z_q^\times$. Accordingly, we have the following non-abelian groups of size $p^2q$ \cite[Proposition 21.17]{EnumerationGroups}:

\begin{itemize}
	\item[(i)] $\mathbb{Z}_{p}\times (\mathbb{Z}_{q}\rtimes_{h^p} \mathbb{Z}_{p})=\langle \sigma,\tau,\epsilon\,|\, \sigma^p=\tau^p=\epsilon^q=1,\ [\epsilon,\tau]=[\tau, \sigma]=1,\, \sigma\epsilon\sigma^{-1}=\epsilon^{h^p}\rangle$,
	\item[(ii)] $\mathbb{Z}_{q}\rtimes_{h^p} \mathbb{Z}_{p^2}=\langle \sigma,\tau\, | \, \tau^q=\sigma^{p^2}=1,\, \sigma\tau\sigma^{-1}=\tau^{h^p} \rangle$,
	\item[(iii)] $\mathbb{Z}_{q}\rtimes_{h} \mathbb{Z}_{p^2}=\langle \sigma,\tau\, | \, \tau^q=\sigma^{p^2}=1,\, \sigma\tau\sigma^{-1}=\tau^h \rangle$.
\end{itemize}

We summarize in the following table the total number of left braces according to the additive and multiplicative isomorphism class of groups that we will obtain in this section.

\begin{table}[H]
	\centering
	\begin{tabular}{c|c|c|c|c|c}
		$+ \backslash \circ$ & $\Z_{p^2q}$& $\Z_q\rtimes_{h^p}\Z_{p^2}$ & $\Z_q\rtimes_h\Z_{p^2}$  & $\Z_p^2\times\Z_q$ & $\Z_p\times(\Z_q\rtimes_{h^p}\Z_p)$ \\
		\hline
		$\Z_{p^2q}$ & $2$ & $p$ & $p$&- &-  \\
		$\Z_p^2\times\Z_q$ &- &- &-& $2$ & $4$ 		
	\end{tabular}
	\caption{Enumeration of left braces of size $p^2q$ with $q=1\pmod{p^2}$.}
\end{table}

\subsection{Left braces of cyclic type}\label{subsection:q=1(p^2)_cyclic}

In this section we denote by $A$ the cyclic group of size $p^2q$. The size of the image of $\pi_2$ of regular subgroups of $\aut{A}$ divides $p^2$.

\begin{lemma}\label{lem:subgroups_p^2_cyclic}
	The subgroups of size $p^2$ of $\aut{A}$ are 
	$$H_j=\langle \varphi_{jp+1,h}\rangle \quad\text{and}\quad T=\langle \varphi_{p+1,1},\ \varphi_{1,h^p} \rangle$$
	for $0\leq j\leq p-1$.
\end{lemma}

\begin{proof}
	Every subgroup of size $p^2$ in $\aut{A}$ embeds into the subgroup of the elements of order at most $p^2$ of $\aut{A}$, which is generated by $\varphi_{p+1,1}$ and $\varphi_{1,h}$ and such group is isomorphic to $\Z_p\times \Z_{p^2}$. According to \cite[Theorem 3.3]{subgroups_finite_p-group}, the group $\Z_p\times \Z_{p^2}$ has $p+1$ subgroups of size $p^2$. The subgroups $\langle \varphi_{p+1,1},\ \varphi_{1,h^p}\rangle$ and $\langle \varphi_{jp+1,h}\rangle$ for $0\leq j\leq p-1$ are $p+1$ distinct subgroups of size $p^2$.
\end{proof}

We will need this lemma in the following proposition.

\begin{lemma}\label{lem:arithmetic_f}
	The mapping 
	%$$f:\mathbb{Z}_{p^2}\longrightarrow \mathbb{Z}_{p^2},\quad m\mapsto \sum_{i=0}^{m-1} (p+1)^i$$
	$$f_j:\mathbb{Z}_{p^2}\longrightarrow \mathbb{Z}_{p^2},\quad m\mapsto \frac{m(m-1)}{2}jp+m$$ 
	is a bijection for every $0\leq j\leq p-1$.
\end{lemma}
\begin{proof}
	Clearly, we have $f_j(m)=m \pmod p$ and	we can prove inductively that $f_j(m+kp)=f_j(m)+kp$. Since every element in $\Z_{p^2}$ is of the form $m+kp$ for suitable $m,k$ then $f_j$ is surjective. 
\end{proof}

\begin{proposition}
	The left braces of cyclic type with $|\ker{\lambda}|=q$ are $(B_j,+,\circ)$ for $0\leq j\leq p-1$ where
	\[\begin{pmatrix}
	n \\ m 
	\end{pmatrix} + \begin{pmatrix}
	x \\ y 
	\end{pmatrix} = \begin{pmatrix}
	n+x \\ m+y \end{pmatrix}, \quad
	\begin{pmatrix}
	n \\ m 
	\end{pmatrix} \circ
	\begin{pmatrix}
	x \\ y 
	\end{pmatrix} =
	\begin{pmatrix}
	n+\big(f_j^{-1}(n)pj+1\big)x \\ m+h^{f_j^{-1}(n)}y 
	\end{pmatrix},
	\]
	for every $0\leq n,x\leq p-1$ and $0\leq m,y\leq q-1$. In particular, $(B_j,\circ)\cong \mathbb{Z}_{q}\rtimes_h \mathbb{Z}_{p^2}$.
\end{proposition}

\begin{proof}
	The groups 
	\begin{equation*}
		G_{j}=\langle \tau,\ \sigma\varphi_{jp+1,h}\rangle \cong \mathbb{Z}_{q}\rtimes_h \mathbb{Z}_{p^2},
	\end{equation*}
	for $0\leq j\leq p-1$ are regular and they are not conjugate since their image under $\pi_2$ are not.
	According to Lemma \ref{lem:subgroups_p^2_cyclic}, we have the following cases:
	%know that the subgroups of size $p^2$ in $\aut{\Z_{p^2q}}$ are of two types. Take $h$ an element of order $p^2$ in $\Z_q^\times$:
	\begin{enumerate}
		\item[(i)] $\pi_2(G)=T$: arguing as in Proposition \ref{prop:cyclic_ker=q_non-existence}, we can show that there are no regular subgroups with this projection.
		
		\item[(ii)] $\pi_2(G)=H_j$: in this case a standard presentation of $G$ is
		\[
		G= \langle \tau,\ \sigma^a\tau^b\varphi_{jp+1,h}\rangle = \langle \tau,\ \sigma^a\varphi_{jp+1,h}\rangle,
		\]
		where $a\neq 0$ since $G$ is regular. Then $G$ is conjugate to $G_j$ by $\varphi_{a^{-1},1}$. 
	\end{enumerate}
	Let $(B_j,+,\circ)$ be the left brace associated to $G_j$, then
	$$\underbrace{\sigma\circ \ldots \circ \sigma}_n=\sigma^{f_j(n)},$$
	where $f_j$ is defined as in Lemma \ref{lem:arithmetic_f}. Therefore $\lambda_{\sigma^n\tau^m}=\lambda_{\tau^m\circ\sigma^n}=\lambda_{\sigma}^{f_j^{-1}(n)}$ and so the formula follows.
\end{proof}

For the case $|\pi_2(G)|=p$, the subgroups of size $p$ of $\aut{A}$ are the same as in the Subsection \ref{subsection:q=1(p)_cyclic} and so we can argue as in Proposition \ref{prop:cyclic_ker=pq_2} and then we have $p+1$ left braces of $A$-type with $|\ker{\lambda}|=pq$.

We summarize the content of this subsection in the following table:

\begin{table}[H]
	\centering
	\begin{tabular}{c|c|c|c}
		$|\ker{\lambda}|$ &  $\mathbb{Z}_{p^2q}$  &  $\Z_q\rtimes_h\Z_{p^2}$ & $\Z_q\rtimes_{h^p}\Z_{p^2}$  \\
		\hline
		$q$ &- & $p$ &- \\
		$pq$ & $1$ &- & $p$ \\
		$p^2q$ & $1$ &- &-
	\end{tabular}
	\caption{Enumeration of left braces of cyclic type of size $p^2q$ for $q=1\pmod{p^2}$.}
\end{table}

\subsection{Left braces of $ \mathbb{Z}_{p}^2\times \mathbb{Z}_{q}$-type}\label{subsection:q=1(p^2)_abelian_non-cyclic}

In this section we denote by $A$ the group $\Z_{p}^2\times \Z_{q}$. The size of the kernel of $\lambda$ of non trivial left braces of $A$-type is either $q$ or $pq$. The conjugacy classes of subgroups of size $p$ of $\aut{A}$ are the same as in the case $q=1\pmod p$ and therefore, if $|\ker{\lambda}|=pq$ we are in the same situation as in Proposition \ref{sub_abelian_p2}.

%\comment{turn this into a Lemma/Prop?}\com{I think this section is fine. We don't need any prop because we are explaining that there are no new cases to consider.}

Up to conjugation, the subgroups of order $p^2$ of $\aut{A}$ are $\langle C, h^p\rangle$ and $ \langle C^l h\rangle$ where $l=0,1$ and $C$ is as in Remark \ref{subgroups of GL}. % and $h$ is identified with the automorphism of $\Z_q$ given by $x\mapsto x^h$ for all $x\in \Z_q$. 
If $G$ is a regular subgroup of $\Hol(A)$ with $\pi_2(G)=\langle C^l h\rangle$ then
$$G=\langle \epsilon,\ v C^l h\rangle $$
for some $v=\sigma^a\tau^b$. Then %group $G$ is not regular, since 
$$(v C^l h)^p=vC^l(v)C^{2l}(v)\ldots C^{(p-1)l}(v) C^{pl}h^p=h^p\in G.$$
So, according to Lemma \ref{regularity} $G$ is not regular.
%$$\pi_1(G)=\setof{\epsilon^n \sigma^{ma+b\frac{m(m-1)}{2}}\tau^{bm}}{0\leq n\leq q-1,\, 0\leq m\leq p-1,}\neq A$$ and so $G$ is not regular. 
Otherwise, if $\pi_2(G)=\langle C, h^p\rangle$, we can argue as in Proposition \ref{sub_abelian_pp2} and therefore it provides a description of skew braces of $A$-type with $|\ker{\lambda}|=q$.

Hence, the enumeration of left braces of $A$-type of size $p^2q$ for $q=1\pmod{p^2}$ is as in Table \ref{table enum abelian not cyclic q=1 (p)}.

\section{Left braces of size $4q$ with $q=1 \pmod{2}$ and $q\neq 1 \pmod{4}$.} \label{4p (i)}

We consider here the case $q>3$ prime. The results contained in this section and in Section \ref{4p (ii)} are contained in the work of Dietzel (cf. \cite[Theorem 5]{Dietzel}).
We assume that $q=1\pmod{2}$ (since it is an odd prime) but $q\neq 1\pmod{4}$. 
According to \cite[Proposition 2.1]{kohl_4q}, the non-abelian groups of size $4q$ for $q>3$ are:
\begin{enumerate}
	\item[(i)] $\Z_2\times (\Z_q\rtimes_{-1}\Z_2)$, the dihedral group of order $4q$, and
	\item[(ii)] $\Z_q\rtimes_{-1}\Z_4 = \langle \sigma, \tau\, |\, \sigma^4=\tau^q=1,\ \sigma\tau\sigma^{-1}=\tau^{-1}\rangle$. 
\end{enumerate}

For $q=3$, the non-abelian groups of order $12$ are $\Z_2\times (\Z_3\rtimes_{-1}\Z_2)$, $\Z_3\rtimes_{-1}\Z_4$ and $A_4$, the alternating group on $4$ letters. The corresponding braces are implemented on the \emph{YangBaxter} package on GAP, \cite{YBE}, so we decided to drop this special case.

The enumeration of left braces we will get in this section is summarized in the following table.

\begin{table}[H]
	\centering
	\begin{tabular}{c|c|c|c|c}
		$+\backslash \circ$ & $\Z_{4q}$ & $\Z_2^2\times\Z_q$ & $\Z_2\times (\Z_q\rtimes_{-1}\Z_2)$ & $\Z_q\rtimes_{-1}\Z_4$ \\
		\hline
		$\Z_{4q}$ & $1$ & $1$ & $2$ & $1$ \\
		$\Z_2^2\times\Z_q$ & $1$ & $1$ & $1$ & $1$
	\end{tabular}
	\caption{Enumeration of left braces of order $4q$ with $q=1\pmod{2}$ and $q\neq 1\pmod{4}$.}
\end{table}

\subsection{Left braces of cyclic type}\label{subsection:4p_i_cyclic} 

In this subsection $A$ will denote the cyclic group of order $4q$.
Since $\aut A$ has order $2(q-1)$, if $G$ is a regular subgroup of $\Hol(A)$ then $|\pi_2(G)|$ divides $4$. The automorphism group has a unique subgroup of order $4$ which is generated by $\varphi_{-1,1}$ and $\varphi_{1,-1}$ and it is isomorphic to $\Z_2\times\Z_2$. So it has three non conjugated subgroups of size $2$, generated by $\varphi_{(-1)^i,(-1)^j}$ for $(i,j)\in \{(1,0),(1,1),(0,1)\}$.

\begin{proposition}\label{prop:4p_i_cyclic_2}
	The left braces of order $4q$ with $|\ker{\lambda}|=2q$ are $(B_{i,j},+,\circ)$ for $(i,j)\in \{(1,0),(1,1),(0,1)\}$ where
	% \begin{table}[H]
	%     \centering
	$$\begin{pmatrix} n \\ m \end{pmatrix} + \begin{pmatrix} x \\ y \end{pmatrix}=\begin{pmatrix} n+x \\ m+y \end{pmatrix}, \quad  \begin{pmatrix} n \\ m \end{pmatrix} \circ \begin{pmatrix} x \\ y \end{pmatrix}=\begin{pmatrix} n+(-1)^{im} x \\ m+(-1)^{jm}y \end{pmatrix} $$
	for every $0\leq n,x\leq q-1$ and $0\leq m,y\leq 3$. In particular,
	$$(B_{i,j},\circ)\cong \begin{cases}
	\Z_q\rtimes_{-1}\Z_4, \quad \text{if }(i,j)=(0,1),\\
	\Z_2\times (\Z_q\rtimes_{-1}\Z_2), \quad \text{if }(i,j)=(1,1),\\
	\Z_2^2\times \Z_q, \quad \text{if }(i,j)=(1,0).\\
	\end{cases}$$
\end{proposition}

\begin{proof}
	The groups
	$$G_{i,j}=\langle \tau,\, \sigma^2,\, \sigma\varphi_{(-1)^i,(-1)^j} \rangle\cong \begin{cases}
	\Z_q\rtimes_{-1}\Z_4, \quad \text{if }(i,j)=(0,1),\\
	\Z_2\times (\Z_q\rtimes_{-1}\Z_2), \quad \text{if }(i,j)=(1,1),\\
	\Z_2^2\times \Z_q, \quad \text{if }(i,j)=(1,0)\\
	\end{cases}$$
	for $(i,j)\in I=\{(1,0),(1,1),(0,1)\}$ are regular and they are not pairwise conjugate since they are not isomorphic.
	
	As we mentioned above, if $G$ is a regular subgroup, then $\pi_2(G)=\langle \varphi_{(-1)^i,(-1)^j}\rangle$ for $(i,j)\in I$. Let assume that $\pi_2(G)$ is generated by $\varphi_{1,-1}$, so $G$ has the following standard presentation:
	\[
	G=\langle \tau,\ \sigma^2,\ \sigma^a\varphi_{1,-1} \rangle. 
	\]
	Since $G$ is regular, then $a\neq 0$. So we can assume that $a=1$ and then $G=G_{0,1}$. Since $\tau\in \ker{\pi_2}$ and $\varphi_{1,-1}(\sigma)=\sigma$, the skew brace $B_{0,1}$ associated to $G_{0,1}$ is given by $\ker{\lambda}+\Fix(B_{0,1})$. Then according to Lemma \ref{ker fix}, we have that $\lambda_{\tau^n \sigma^m}=\lambda_{\tau^n\circ \sigma^m}=\lambda_{\sigma^m}=\lambda_{\sigma}^m$. Therefore the  formula for $B_{0,1}$ follows. 
	
	The other two cases are analogous, so we skip the computations.
\end{proof}

\begin{lemma}\label{lem:4p_i_cyclic_4}
	The unique left brace of cyclic type with $|\ker{\lambda}|=q$ is $(B,+,\circ)$ where
	\[
	\begin{pmatrix}
	n \\ m
	\end{pmatrix} +\begin{pmatrix}
	x \\ y
	\end{pmatrix}=
	\begin{pmatrix}
	n+x \\ m+y
	\end{pmatrix}, \qquad
	\begin{pmatrix}
	n \\ m
	\end{pmatrix} \circ \begin{pmatrix}
	x \\ y
	\end{pmatrix}=
	\begin{pmatrix}
	n+(-1)^{\frac{m(m-1)}{2}}x \\
	m+(-1)^m y
	\end{pmatrix}
	\]
	where $0\leq n,x\leq q-1$ and $0\leq m,y\leq 3$. In particular, $(B,\circ)\cong \Z_2\times (\Z_q\rtimes_{-1}\Z_2)$.
\end{lemma}

\begin{proof}
	The group
	\[H=
	\langle \tau,\ \sigma^2\varphi_{1,-1},\ \sigma\varphi_{-1,1}\rangle \cong \Z_2\times (\Z_q\rtimes_{-1}\Z_2)
	\]
	is regular. Let $G$ be a regular subgroup of $\Hol(A)$ with $|\pi_2(G)|=4$. We show that it is conjugated to $H$. Since the only subgroup of size $q$ of $A$ is generated by $\tau$, we have that $G$ has the following standard presentation
	\[
	G=\langle \tau,\ \sigma^a\varphi_{1,-1},\ \sigma^b\varphi_{-1,1}\rangle
	\]
	for some $1\leq a,b\leq 3$. By checking the conditions \NC we have that $a=2$. If $b=2$ then $\sigma^2 \varphi_{1,-1}\sigma^2\varphi_{-1,1}=\varphi_{1,-1}\varphi_{-1,1}\in G$ and so $G$ is not regular. Hence, up to conjugation by $\varphi_{-1,1}$, we can assume that $b=1$, and so $G$ is conjugated to $H$.
	The formulas in the statement define a left brace with the desired properties, and therefore it is isomorphic to the one associated to the unique regular subgroup $G$ with $|\pi_2(G)|=4$.  
\end{proof}

We have the following table with the left braces of cyclic type:

\begin{table}[H]
	\centering
	\begin{tabular}{c|c|c|c|c}
		$\ker{\lambda}$ & $\Z_{4q}$ & $\Z_2^2\times\Z_q$ & $\Z_2\times (\Z_q\rtimes_{-1}\Z_2)$ & $\Z_q\rtimes_{-1}\Z_4$\\
		\hline
		$q$ &- &- & $1$ &- \\
		$2q$ &- & $1$ & $1$ & $1$ \\
		$4q$ & $1$ &- &- &-
	\end{tabular}
	\caption{Enumeration of left braces of cyclic type of size $4q$ with $q=1\pmod{2}$ and $q\neq 1\pmod{4}$.}
	\label{table:cyclic_4p_i}
\end{table}

\subsection{Left braces of $\Z_2^2\times\Z_q$-type.}\label{subsection:4p_i_abelian} In this subsection, $A$ will denote the abelian group $\Z_2^2\times\Z_q$.
Up to conjugation, $\aut{A}$ has a unique subgroup of order $4$ which is isomorphic to $\Z_2\times\Z_2$ and it is generated by $$C=\begin{pmatrix}
1 & 1 \\ 0 & 1
\end{pmatrix}$$ and $\eta=-1$. So, there are three possible subgroups of order $2$ up to conjugation, namely the subgroups generated by $C^i\eta^j$ for $(i,j)\in \{(1,0),(1,1),(0,1)$.

\begin{lemma}\label{lem:4p_i_abelian_4}
	If $G$ is a regular subgroup of $\Hol(A)$, then $|\pi_2(G)|\neq 4$.
\end{lemma}
\begin{proof}
	By the discussion above and the fact that $A$ has an unique subgroup of order $q$, say $\langle\epsilon\rangle$, then $G$ has the following standard presentation:
	\[
	G=\langle \epsilon,\ \sigma^a\tau^b C,\ \sigma^c\tau^d\eta \rangle
	\]
	for some $0\leq a,b,c,d\leq 1$. From the \NC conditions $(\sigma^a\tau^b C)^2, [\sigma^a\tau^b C,\sigma^c\tau^d\eta]\in \langle \epsilon\rangle$ then we have $b=d=0$ and so $a,c\neq 0$. Hence, $(\sigma C)^{-\frac{c}{a}} \sigma^c\eta=C^{-\frac{c}{a}}\eta\in G$ and so $G$ is not regular.
\end{proof}

\begin{proposition}\label{lem:4p_i_abelian_2}
	The left braces of order $4q$ with $|\ker{\lambda}|=2q$ are $(B_{i,j},+,\circ)$ for $(i,j)\in \{(1,0),(0,1),(1,1)\}$, where
	$$\begin{pmatrix}
	n \\ m \\l 
	\end{pmatrix} +\begin{pmatrix}
	x \\ y \\z
	\end{pmatrix}=
	\begin{pmatrix}
	n+x \\ m+y \\ l+z
	\end{pmatrix}, \qquad \begin{pmatrix} n \\ m \\ l \end{pmatrix} \circ \begin{pmatrix} x \\ y \\ z \end{pmatrix}=\begin{pmatrix} n+(-1)^{il} x \\ m+y+j l z \\ l+z \end{pmatrix}$$
	for $0\leq n,x\leq q-1$ and $0\leq m,l,y,z\leq 1$. In particular,
	$$(B_{i,j},\circ)\cong \begin{cases}
	\Z_2\times (\Z_q\rtimes_{-1}\Z_2), \quad \text{if }(i,j)=(1,0),\\
	\Z_q\rtimes_{-1}\Z_4, \quad \text{if }(i,j)=(1,1),\\
	\Z_{4q}, \quad \text{if }(i,j)=(0,1).\\
	\end{cases}$$
\end{proposition}

\begin{proof}
	The groups 
	\begin{align*}
		G_{i,j}=\langle \sigma,\ \epsilon,\ \tau\eta^i C^j\rangle\cong \begin{cases}
			\Z_2\times (\Z_q\rtimes_{-1}\Z_2), \quad \text{if }(i,j)=(1,0),\\
			\Z_q\rtimes_{-1}\Z_4, \quad \text{if }(i,j)=(1,1),\\
			\Z_{4q}, \quad \text{if }(i,j)=(0,1)\\
		\end{cases}
	\end{align*}
	for $(i,j)\in  \{(1,0),(0,1),(1,1)\}$ are regular and they are not pairwise conjugate since they are not isomorphic. 
	
	Let $G$ be a regular subgroup of $\Hol(A)$ with $|\pi_2(G)|=2$ and the kernel has size $2q$, so it has the form $\langle \epsilon, v\rangle$ for $v\in \langle \sigma,  \tau\rangle$.
	
	Suppose that $\pi_2(G)=\langle \eta \rangle$. Since every element in $GL_2(2)$ centralizes $\eta$ we have that $v=\sigma$ up to conjugation. So $G$ is conjugated to $G_{1,0}$. Let $B_{1,0}$ denotes the left brace associated to $G_{1,0}$. Since $\langle \sigma, \epsilon\rangle\subseteq\ker{\lambda}$ and $\langle\tau\rangle\subseteq\Fix(B_{1,0})$, we have that $B_{1,0}=\ker{\lambda}+\Fix(B_{1,0})$. Hence, using Lemma \ref{ker fix}, the formula follows.

	For the case $\pi_2(G)=\langle C\rangle$, by conjugating by elements in $N_{GL_2(p)})(C)$ we can assume that $v=\sigma,\tau$. The first one gives us $G_{0,1}$ and the second leads to the standard presentation
	\[
	G=    \langle \tau,\ \epsilon,\ \sigma C \rangle.
	\]
	Condition \NC shows a contradiction. 
	
	If $\pi_2(G)=\langle C,\eta\rangle$ we can argue in the same fashion and also the formulas follow by the same argument as the first case.
\end{proof}

The following table summarizes the results of this subsection.

\begin{table}[H]
	\centering
	\begin{tabular}{c|c|c|c|c}
		$\ker{\lambda}$ & $\Z_{4q}$ & $\Z_2^2\times\Z_q$ & $\Z_2\times (\Z_q\rtimes_{-1}\Z_2)$ & $\Z_q\rtimes_{-1}\Z_4$  \\
		\hline
		$2q$ & $1$ &- & $1$& $1$ \\
		$4q$ &- & $1$ &- &-
	\end{tabular}
	\caption{Enumeration of left braces of size $\Z_2^2\times\Z_q$-type.}
	\label{table:abelian_4p}
\end{table}

\section{Left braces of size $4q$ with $q=1 \pmod{4}$.} \label{4p (ii)}

In this section we consider the skew braces of order $4q$ for $q=1\pmod{4}$. The non-abelian groups of order $4q$ are the same as in Section \ref{4p (i)} but we have one extra group given by:
$$\Z_q\rtimes_{\xi}\Z_4=\langle \sigma,\tau\, |\, \sigma^4=\tau^q=1,\ \sigma\tau\sigma^{-1}=\tau^\xi \rangle$$
where $\xi$ is an element of order $4$ in $\Z_q^\times$, \cite[Proposition 2.1]{kohl_4q}.

The enumeration of left braces we will get in this section is summarized in the following table.

\begin{table}[H]
	\centering
	\begin{tabular}{c|c|c|c|c|c}
		$+\backslash \circ$ & $\Z_{4q}$ & $\Z_2^2\times\Z_q$ & $\Z_2\times (\Z_q\rtimes_{-1}\Z_2)$ & $\Z_q\rtimes_{-1}\Z_4$ & $\Z_q\rtimes_{\xi}\Z_4$ \\
		\hline
		$\Z_{4q}$ & $1$ & $1$ & $2$ & $1$ & $1$ \\
		$\Z_2^2\times\Z_q$ & $1$ & $1$ & $1$ & $1$ & $1$ 
	\end{tabular}
	\caption{Enumeration of left braces of order $4q$ with $q=1\pmod{4}$.}% according to their additive and multiplicative group.}
\end{table}

\subsection{Left braces of cyclic type.}
In this subsection we denote by $A$ the cyclic group of order $4q$. The automorphism group of $A$ has size $2(q-1)$, so the size of the image of every regular subgroup under $\pi_2$ divides $4$.
The group $\Z_q^\times$ contains an element $\xi$ of order $4$ since $q=1\pmod 4$ and clearly $\xi^2=-1$.

The subgroup of $\aut A\cong \Z_2\times \Z_q^\times$ containing the elements of order at most $4$ is generated by $\varphi_{-1,1}$ and $\varphi_{1,\xi}$ and it is isomorphic to $\Z_2\times\Z_4$. The subgroup of elements of order $2$ of $\aut{A}$ is generated by $\varphi_{-1,1}$ and $\varphi_{1,-1}$, so it is the same we considered in Subsection \ref{subsection:4p_i_cyclic}, and then the left braces with $|\ker{\lambda}|=2q$ are as in Lemma \ref{prop:4p_i_cyclic_2}.

On the other hand, we have three subgroups of $\aut{A}$ of order $4$: $\langle \varphi_{1,-1}, \varphi_{-1,1}\rangle$, $\langle \varphi_{1,\xi}\rangle$ and $\langle \varphi_{-1,\xi}\rangle$.

\begin{lemma}
	The left braces of size $4q$ and $|\ker{\lambda}|=q$ are $(B_i,+,\circ)$ for $i=1,2$ where
	\begin{align*}
		B_1: &\quad \begin{pmatrix}
			n \\ m  
		\end{pmatrix} +\begin{pmatrix}
			x \\ y 
		\end{pmatrix}=
		\begin{pmatrix}
			n+x \\ m+y 
		\end{pmatrix}, \qquad \begin{pmatrix} n \\ m \end{pmatrix} \circ \begin{pmatrix} x \\ y  \end{pmatrix}=\begin{pmatrix}
			n+(-1)^{\frac{m(m-1)}{2}}x \\
			m+(-1)^m y
		\end{pmatrix},\\
		B_2: & \quad \begin{pmatrix}
			n \\ m  
		\end{pmatrix} +\begin{pmatrix}
			x \\ y 
		\end{pmatrix}=
		\begin{pmatrix}
			n+x \\ m+y 
		\end{pmatrix}, \qquad \begin{pmatrix} n \\ m \end{pmatrix} \circ \begin{pmatrix} x \\ y  \end{pmatrix}=\begin{pmatrix}
			n+\xi^{m}x \\
			m+ y
		\end{pmatrix}
	\end{align*}
	for $0\leq n,x\leq q-1$ and $0\leq m,y\leq 3$. In particular $(B_2,+,\circ)$ is a bi-skew brace and we have
	$$(B_1,\circ)\cong \Z_2\times (\Z_q\rtimes_{-1}\Z_2) \qquad\text{and}\qquad (B_2,\circ)\cong \Z_q\rtimes_{\xi}\Z_4.$$
\end{lemma}
\begin{proof}
	The groups
	\[
	G_1=\langle \tau,\ \sigma^2\varphi_{1,-1},\ \sigma\varphi_{-1,1}\rangle \cong \Z_2\times (\Z_q\rtimes_{-1}\Z_2) \qquad \text{and} \qquad G_2=\langle \tau,\ \sigma\varphi_{1,\xi} \rangle \cong \Z_q\rtimes_{\xi}\Z_4
	\]
	are regular and they are not conjugate, since they are not isomorphic.
	Let $G$ be a regular subgroup of $\Hol(A)$ with $|\pi_2(G)|=4$. The group $A$ has a unique subgroup of order $q$ which is generated by $\tau$. We have three cases. If $\pi_2(G)=\langle \varphi_{1,-1},\ \varphi_{-1,1}\rangle$, then Lemma \ref{lem:4p_i_cyclic_4} applies and we get the brace $B_1$. If $\pi_2(G)=\langle\varphi_{1,\xi}\rangle$, then a standard presentation for $G$ is $G_2$, the associate brace is $B_2$ and according to \cite[Proposition 1.1]{skew_pq} $B_2$ is a bi-skew brace. If $\pi_2(G)=\langle\varphi_{-1,\xi}\rangle$, then a standard presentation for $G$ is
	\[
	G=\langle \tau,\ \sigma\varphi_{-1,\xi} \rangle.
	\]
	Since $(\sigma\varphi_{-1,\xi})^{2}=\varphi_{-1,\xi}^2=\varphi_{1,-1}\in G$, then $G$ is not regular by Lemma \ref{regularity}.
\end{proof}

The following table summarizes the left braces we obtained in this subsection.

\begin{table}[H]
	\centering
	\begin{tabular}{c|c|c|c|c|c}
		$\ker{\lambda}$ & $\Z_{4q}$ & $\Z_2^2\times\Z_q$ & $\Z_2\times (\Z_q\rtimes_{-1}\Z_2)$ & $\Z_q\rtimes_{-1}\Z_4$ & $\Z_q\rtimes_{\xi}\Z_4$ \\
		\hline
		$q$ &- &- & $1$ &- & $1$ \\
		$2q$ &- & $1$ & $1$ & $1$ &- \\
		$4q$ & $1$ &- &- &- &-
	\end{tabular}
	\caption{Enumeration of left braces of cyclic type of size $4q$ with $q=1\pmod{4}$.}
	\label{table:cyclic_4p_ii}
\end{table}

\subsection{Left braces of $\Z_2^2\times\Z_q$-type.} In this section $A$ will denote the group $\Z_2^2\times\Z_q$. As in Subsection \ref{subsection:4p_i_abelian} we consider the subgroup of $\aut{A}\cong GL_2(2)\times \Z_q^\times$ containing the elements of order at most $4$, which is generated by $C$ and $\xi$. Such subgroup is isomorphic to $\Z_2\times\Z_4$ and it has three subgroups of order $4$.

If $G$ is a regular subgroup with $|\pi_2(G)|=2$, then Lemma \ref{lem:4p_i_abelian_2} applies.

\begin{lemma}
	The unique left brace of size $4q$ with $|\ker{\lambda}|=q$ is $(B,+,\circ)$ where
	$$ \begin{pmatrix}
	n \\ m \\l 
	\end{pmatrix} +\begin{pmatrix}
	x \\ y \\z
	\end{pmatrix}=
	\begin{pmatrix}
	n+x \\ m+y \\ l+z
	\end{pmatrix}, \qquad \begin{pmatrix} n \\ m \\ l \end{pmatrix} \circ \begin{pmatrix} x \\ y \\ z \end{pmatrix}=\begin{pmatrix}
	n+\xi^{2m+l} x \\ m+y+lz \\ l+z
	\end{pmatrix}$$
	for $0\leq n,x\leq q-1$ and $0\leq m,l,y,z\leq 1$. In particular, $(B,\circ)\cong \Z_q\rtimes_{\xi}\Z_4$.
\end{lemma}    

\begin{proof}
	If $G$ is a regular subgroup of $\Hol(A)$ with $|\pi_2(G)|=4$, then the kernel is generated by $\epsilon$. If $\pi_2(G)=\langle C, \xi^2\rangle$, then the same argument of Lemma \ref{lem:4p_i_abelian_4} applies and we get a contradiction. If $\pi_2(G)=\langle\xi\rangle$, then $G$ has the following standard presentation
	\[
	G=\langle \epsilon,\ \sigma^a\tau^b\xi \rangle.
	\]
	Then $(\sigma^a\tau^b\xi)^2=\xi^2\in G$ and so $G$ is not regular. 
	If $\pi_2(G)=\langle C\xi\rangle$, then $G$ has the following standard presentation
	\[
	G=\langle \epsilon,\ \sigma^a\tau^b C\xi \rangle\cong \Z_q\rtimes_{\xi}\Z_4.
	\]
	If $b=0$ then $(\sigma^a C\xi)^2=\xi ^2\in G$ and so $G$ is not regular. Then $b=1$ and, up to conjugation by $C$, we can assume that $a=0$. The group $G$ is regular. Hence, in the associated brace $B$, we have $\tau\circ \tau=\tau C(\tau)= \sigma$ and so $\lambda_{\sigma}=\lambda_{\tau}^2=\xi^2$. Then 
	$$\lambda_{\epsilon^n \sigma^m \tau^l}=\lambda_{\epsilon^n \circ \sigma^m\circ  \tau^l}=\lambda_{\sigma}^m \lambda_{\tau}^l$$
	since $l=0,1$ and the formula follows.
\end{proof}

The following table summarizes the left braces we obtained in this subsection.

\begin{table}[H]
	\centering
	\begin{tabular}{c|c|c|c|c|c}
		$\ker{\lambda}$ & $\Z_{4q}$ & $\Z_2^2\times\Z_q$ & $\Z_2\times (\Z_q\rtimes_{-1}\Z_2)$ & $\Z_q\rtimes_{-1}\Z_4$ & $\Z_q\rtimes_{\xi}\Z_4$ \\
		\hline
		$q$ &- &- &- &- & $1$ \\
		$2q$ & $1$ &- & $1$ & $1$ &- \\
		$4q$ &- & $1$ &- &- &-
	\end{tabular}
	\caption{Enumeration of left braces of $\Z_2^2\times\Z_q$-type of size $4q$ with $q=1\pmod{4}$.}
	\label{table:abelian_4p_ii}
\end{table}

\section{Left braces of size $p^2q$ with $p,q$ arithmetically independent}\label{section:alg_ind}

Let $p,q$ be primes. If none of the following congruences holds
\begin{equation*}\label{congruences}
	p=1\pmod{q},\quad p=-1\pmod{q},\quad q=1\pmod{p},    
\end{equation*}
then $p$ and $q$ are said to be \emph{arithmetically independent}. In such case, the only groups of size $p^2q$ are the abelian ones, $\Z_{p^2q}$ and $\Z_p^2\times \Z_q$, \cite[Proposition 21.17]{EnumerationGroups}. In both cases, the size of the kernel of $\lambda$ of any non-trivial brace is $pq$. Applying Theorems \ref{prop:cyclic_ker=pq} and \ref{sub_abelian_p} we have the following table:
\begin{table}[H]
	\centering
	\begin{tabular}{c|c|c}
		$+\backslash \circ$ & $\Z_{p^2q}$ & $\Z_p^2\times\Z_q$  \\
		\hline
		$\Z_{p^2q}$ & $2$ &- \\
		$\Z_p^2\times\Z_q$ &- & $2$
	\end{tabular}
	\caption{Enumeration of left braces of size $p^2q$ with $p$ and $q$ arithmetically independent.}
\end{table}

\section*{Acknowledgement}
This work was partially supported by UBACyT 20020171000256BA and PICT 2016-2481. The authors want to thank Manoj Yadav and Leandro Vendramin for their comments on an earlier version of this work.

\bibliographystyle{abbrv}
\bibliography{refs}

\def\cprime{$'$}
\begin{thebibliography}{10}

\bibitem{skew_pq}
E.~Acri and M.~Bonatto.
\newblock Skew braces of size $pq$.
\newblock {\em Communications in Algebra}, 48(5):1872--1881, 2020.

\bibitem{second_paper}
E.~Acri and M.~Bonatto.
\newblock Skew braces of size $p^2q$ {I}{I}: Non-abelian type.
\newblock {\em Journal of Algebra and Its Applications}, 21(03):2250062, 2022.

\bibitem{squarefree}
A.~A. {Alabdali} and N.~P. {Byott}.
\newblock Skew braces of squarefree order.
\newblock {\em Accepted for publication in Journal of Algebra And Its
  Applications}, doi:10.1142/S0219498821501280.

\bibitem{p_cube}
D.~Bachiller.
\newblock Classification of braces of order {$p^3$}.
\newblock {\em J. Pure Appl. Algebra}, 219(8):3568--3603, 2015.

\bibitem{MR3835326}
D.~Bachiller.
\newblock Solutions of the {Y}ang-{B}axter equation associated to skew left
  braces, with applications to racks.
\newblock {\em J. Knot Theory Ramifications}, 27(8):1850055, 36, 2018.

\bibitem{MR3527540}
D.~Bachiller, F.~Ced\'{o}, and E.~Jespers.
\newblock Solutions of the {Y}ang-{B}axter equation associated with a left
  brace.
\newblock {\em J. Algebra}, 463:80--102, 2016.

\bibitem{skew_trick}
V.~G. Bardakov, M.~V. Neshchadim, and M.~K. Yadav.
\newblock Computing skew left braces of small orders.
\newblock {\em International Journal of Algebra and Computation},
  30(04):839--851, 2020.

\bibitem{EnumerationGroups}
S.~R. Blackburn, P.~M. Neumann, and G.~Venkataraman.
\newblock {\em Enumeration of finite groups}, volume 173 of {\em Cambridge
  Tracts in Mathematics}.
\newblock Cambridge University Press, Cambridge, 2007.

\bibitem{Caranti}
E.~Campedel, A.~Caranti, and I.~{Del Corso}.
\newblock {H}opf-{G}alois structures on extensions of degree $p^2q$ and skew
  braces of order $p^2q$: The cyclic {S}ylow $p$-subgroup case.
\newblock {\em Journal of Algebra}, 556:1165--1210, 2020.

\bibitem{auto_pq}
E.~Campedel, A.~Caranti, and I.~{Del Corso}.
\newblock The automorphism groups of groups of order $p^2q$.
\newblock {\em International Journal of Group Theory}, 10:149--157, 2021.

\bibitem{nilpotent_type}
F.~Ced\'{o}, A.~Smoktunowicz, and L.~Vendramin.
\newblock Skew left braces of nilpotent type.
\newblock {\em Proc. Lond. Math. Soc. (3)}, 118(6):1367--1392, 2019.

\bibitem{biskew}
L.~N. {Childs}.
\newblock {Bi-skew braces and Hopf Galois structures.}
\newblock {\em {New York J. Math.}}, 25:574--588, 2019.

\bibitem{Crespo_2p2}
T.~Crespo.
\newblock {H}opf {G}alois structures on field extensions of degree twice an odd
  prime square and their associated skew left braces.
\newblock {\em Journal of Algebra}, 565:282--308, 2021.

\bibitem{Dietzel}
C.~Dietzel.
\newblock Braces of order $p^2q$.
\newblock {\em Accepted for publication in Journal of Algebra and Its
  Applications}, doi:10.1142/S0219498821501401.

\bibitem{MR1183474}
V.~G. Drinfel{\cprime}d.
\newblock On some unsolved problems in quantum group theory.
\newblock In {\em Quantum groups ({L}eningrad, 1990)}, volume 1510 of {\em
  Lecture Notes in Math.}, pages 1--8. Springer, Berlin, 1992.

\bibitem{MR1722951}
P.~Etingof, T.~Schedler, and A.~Soloviev.
\newblock Set-theoretical solutions to the quantum {Y}ang-{B}axter equation.
\newblock {\em Duke Math. J.}, 100(2):169--209, 1999.

\bibitem{MR1637256}
T.~Gateva-Ivanova and M.~Van~den Bergh.
\newblock Semigroups of {$I$}-type.
\newblock {\em J. Algebra}, 206(1):97--112, 1998.

\bibitem{MR3647970}
L.~Guarnieri and L.~Vendramin.
\newblock Skew braces and the {Y}ang--{B}axter equation.
\newblock {\em Math. Comp.}, 86(307):2519--2534, 2017.

\bibitem{kohl_4q}
T.~Kohl.
\newblock Groups of order $4p$, twisted wreath products and {H}opf--{G}alois
  theory.
\newblock {\em Journal of Algebra}, 314, 04 2007.

\bibitem{MR1769723}
J.-H. Lu, M.~Yan, and Y.-C. Zhu.
\newblock On the set-theoretical {Y}ang-{B}axter equation.
\newblock {\em Duke Math. J.}, 104(1):1--18, 2000.

\bibitem{NZ}
K.~{Nejabati Zenouz}.
\newblock {\em On Hopf-Galois Structures and Skew Braces of Order $p^{3}$}.
\newblock PhD thesis, The University of Exeter,
  https://ore.exeter.ac.uk/repository/handle/10871/32248, 2018.

\bibitem{NZpaper}
K.~Nejabati~Zenouz.
\newblock Skew braces and {H}opf-{G}alois structures of {H}eisenberg type.
\newblock {\em J. Algebra}, 524:187--225, 2019.

\bibitem{MR2132760}
W.~Rump.
\newblock A decomposition theorem for square-free unitary solutions of the
  quantum {Y}ang-{B}axter equation.
\newblock {\em Adv. Math.}, 193(1):40--55, 2005.

\bibitem{MR2298848}
W.~Rump.
\newblock Classification of cyclic braces.
\newblock {\em J. Pure Appl. Algebra}, 209(3):671--685, 2007.

\bibitem{Rump}
W.~Rump.
\newblock Classification of cyclic braces, {I}{I}.
\newblock {\em Transactions of the American Mathematical Society}, page~1, 03
  2018.

\bibitem{Smoktunowicz}
A.~Smoktunowicz.
\newblock A note on set-theoretic solutions of the {Y}ang–{B}axter equation.
\newblock {\em Journal of Algebra}, 500:3 -- 18, 2018.
\newblock Special Issue dedicated to Efim Zelmanov.

\bibitem{Leandro-Byott}
A.~{Smoktunowicz} and L.~{Vendramin}.
\newblock {On skew braces (with an appendix by N. Byott and L. Vendramin).}
\newblock {\em {J. Comb. Algebra}}, 2(1):47--86, 2018.

\bibitem{MR1809284}
A.~Soloviev.
\newblock Non-unitary set-theoretical solutions to the quantum {Y}ang-{B}axter
  equation.
\newblock {\em Math. Res. Lett.}, 7(5-6):577--596, 2000.

\bibitem{subgroups_finite_p-group}
M.~T\u{a}rn\u{a}uceanu.
\newblock An arithmetic method of counting the subgroups of a finite abelian
  group.
\newblock {\em Bulletin math\'ematique de la Soci\'et\'e des Sciences
  Math\'ematiques de Roumanie}, 53 (101)(4):373--386, 2010.

\bibitem{problems}
L.~Vendramin.
\newblock Problems on skew left braces.
\newblock {\em Advances in Group Theory and Applications}, 7:15--37, 2019.

\bibitem{YBE}
L.~Vendramin and A.~Konovalov.
\newblock Yangbaxter, combinatorial solutions for the {Y}ang-{B}axter equation,
  version 0.9.0 (2019).
\newblock Available at https://gap-packages.github.io/YangBaxter.

\end{thebibliography}
 
 \end{document}